\documentclass[12pt]{amsart}
\pdfoutput=1
\usepackage{geometry}
\usepackage{comment}
\usepackage{fancyhdr}
\usepackage{etoolbox}
\makeatletter

\def\@secnumfont{\bfseries}
\def\section{\@startsection{section}{1}%
  \z@{.7\linespacing\@plus\linespacing}{.5\linespacing}%
  {\bfseries\Large\scshape\centering}}
\def\subsection{\@startsection{subsection}{2}%
  \z@{.5\linespacing\@plus.7\linespacing}{-.5em}%
  {\normalfont\bfseries}}  
\patchcmd{\@thm}{\thm@headfont{\scshape}}{\thm@headfont{\bfseries}}{}{}
\makeatother
\usepackage[T1]{fontenc}
\usepackage[all]{xy}
\usepackage{pdfsync}
\usepackage{bbm}
\usepackage[colorlinks,linkcolor=blue,citecolor=blue,urlcolor=cyan]{hyperref}

\usepackage{tikz-cd}
\usepackage{amsmath}
\usepackage{tikz}
\usepackage[capitalize]{cleveref} 
\usepackage{verbatim}
\usepackage[utf8]{inputenc} 
\usepackage[T1]{fontenc} 
\usepackage{enumerate,comment,braket,xspace,csquotes} 
\newcommand{\Z}{\mathbb{Z}}

\usepackage{mathrsfs}
\usepackage{amssymb,latexsym,amscd,verbatim,color}
\usepackage{bbm}
\usepackage{amsmath}
\usepackage{tikz}
\usepackage[capitalize]{cleveref} 
\usepackage[T1]{fontenc} 
\usepackage{enumerate,comment,braket,xspace,csquotes} 

\newtheorem{thm}{Theorem}
\newtheorem{prop}{Proposition}[section]
\newtheorem{example}{Example}[section]
\newtheorem{lemma}{Lemma}[section]
\newtheorem{cor}{Corollary}[section]

\theoremstyle{definition}
\newtheorem{defn}{Definition}[section]

\newtheorem{remark}{Remark}[section]

\newtheorem*{thm*}{Theorem}
\numberwithin{equation}{section}
\newcounter{spec}
{\end{list}}

\newcommand{\Spec}{\operatorname{Spec}} 
\newcommand{\Spf}{\operatorname{Spf}} 
\newcommand{\Hom}{\operatorname{Hom}}      
\newcommand{\Ext}{\operatorname{Ext}}      
\newcommand{\iext}{\operatorname{\underline{Ext}}}
\newcommand{\Hp}{\operatorname{\cH_{p}^{\varphi}}}
\newcommand{\Hpp}{\operatorname{\cH_{p}^{\varpi}}}
\newcommand{\hp}{\operatorname{h_p}}   
\newcommand{\W}{\operatorname{W}}
\newcommand{\D}{\operatorname{\mathbb{D}}}
\newcommand{\K}{\operatorname{\mathbb{K}}}
          
\newcommand{\Mi}{\mathcal{M}_1}   
\newcommand{\gMi}{\mathcal{M}^{\text{gr}}_1} 

\newcommand{\Tcrys}{\operatorname{T_{crys}}} 
\newcommand{\Tcrysv}{\operatorname{T^{crys}}}

\newcommand{\TdR}{\operatorname{T_{dR}}} 
\newcommand{\TdRv}{\operatorname{T_{dR}^{\vee}}} 
\newcommand{\Tsing}{\operatorname{T_{sing}}} 
\newcommand{\Tp}{\operatorname{T_{p}}} 
\newcommand{\Tl}{\operatorname{T_{\ell}}} 

\newcommand{\mup}{\mu_{p^{\infty}}} 
\newcommand{\ocp}{\cO_{\C_p}}

\newcommand{\Mod}{{\operatorname{Mod\text{}}}}
\newcommand{\Vect}{{\operatorname{Vect\text{}}}}

\newcommand{\Vp}{\operatorname{V_p}}
\newcommand{\Ainf}{\operatorname{A_{inf}}}
\newcommand{\BdR}{\operatorname{B_{dR}}}
\newcommand{\BdRp}{\operatorname{B^+_{dR}}}

\newcommand{\Bcris}{\operatorname{B_{cris}}}
\newcommand{\Bcrisp}{\operatorname{B^+_{cris}}}
\newcommand{\At}{\operatorname{A_2}}
\newcommand{\Bt}{\operatorname{B_2}}
\newcommand{\Bct}{\operatorname{B_{2,cris}}}
\newcommand{\Brigt}{\operatorname{\tilde{B}_{2,rig}}} 

\newcommand{\DdR}{\operatorname{D_{dR}}}
\newcommand{\Dcris}{\operatorname{D_{cris}}}

\newcommand{\Ann}{\operatorname{Ann}}  

\newcommand{\Div}{\operatorname{Div}}

\newcommand{\ihom}{{\rm\underline{Hom}}}  

\newcommand{\C}{\mathbb{C}}     
\newcommand{\F}{\mathbb{F}}     
\newcommand{\Q}{\mathbb{Q}}     
\renewcommand{\Z}{\mathbb{Z}}     

\newcommand{\G}{\mathbb{G}}     

\newcommand{\ve}{^{\vee}} 

\newcommand{\im}{\operatorname{Im}}        
\renewcommand{\ker}{\operatorname{Ker}}  
\newcommand{\coker}{\operatorname{Coker}} 
\newcommand{\gr}{\operatorname{gr}}        
\newcommand{\Pic}{\operatorname{Pic}}     
\newcommand{\Alb}{\operatorname{Alb}}     

\newcommand{\rank}{\operatorname{rank}}    

\newcommand{\eprooff}{\hfill$\Box$\par}
\renewcommand{\qed}{\eprooff}
\renewcommand{\tilde}{\widetilde}

\newcommand{\ie}{{\it i.e., }}

\newcommand{\Lie}{{\operatorname{Lie}}} 
\newcommand{\coLie}{{\operatorname{coLie}}} 
\newcommand{\Fil}{{\operatorname{Fil}}} 
\newcommand{\Ex}{{\operatorname{E}}}
\newcommand{\nat}{^{\natural}}

\renewcommand{\bar}{\overline}
\newcommand{\barQ}{\overline{\Q}} 
 
\newcommand{\into}{\hookrightarrow}

\renewcommand{\lim}{\varprojlim}
\newcommand{\colim}{\varinjlim}
\newcommand{\onto}{\mbox{$\to\!\!\!\!\to$}}

\newcommand{\cA}{\mathcal{A}}

\newcommand{\cC}{\mathcal{C}}

\newcommand{\cH}{\mathcal{H}}

\newcommand{\cK}{\mathcal{K}}

\newcommand{\cO}{\mathcal{O}}
\newcommand{\cP}{\mathcal{P}}

\newcommand{\sG}{\mathscr{G}}
\newcommand{\sI}{\mathscr{I}}

\newcommand{\fg}{\mathfrak{g}}

\newcommand{\fh}{\mathfrak{h}}

\newcommand{\fm}{\mathfrak{m}}
\newcommand{\fb}{\mathfrak{b}}
\newcommand{\fa}{\mathfrak{a}}

\renewcommand{\phi}{\varphi}
\renewcommand{\epsilon}{\varepsilon}
\newcommand{\boxtensor}{\def\boxtimesten{\Box\kern-7.59pt\raise1.2pt
\hbox{$\times$} }}                                  

\newcounter{elno}                   

\begin{document}

\title{p-adic period conjectures for  1-motives:\\[2ex]
     Integration and Linear Relations}
\author{Mohammadreza Mohajer} \author{Abdellah Sebbar}
\address{Department of Mathematics and Statistics, University of Ottawa,
		Ottawa Ontario K1N 6N5 Canada}
	\email{mmohajer@uottawa.ca}
	\email{asebbar@uottawa.ca}



\begin{abstract}
We develop a $p$-adic theory of periods for 1-motives, extending the classical theory of complex periods into the non-archimedean setting. For 1-motives with good reduction over $p$-adic local fields, we construct a $p$-adic integration pairing that generalizes the Colmez--Fontaine--Messing theory for abelian varieties. This pairing is bilinear, perfect, Galois-equivariant, and compatible with the Hodge filtration, taking values in a quotient of the de Rham period ring.

Building on this construction, we introduce a stratified formalism for $p$-adic periods, defining period spaces at various depths that capture increasingly refined relations among periods, and formulating conjectures that mirror the Grothendieck period conjecture in this new context. The classical period conjecture for 1-motives over $\barQ$, previously resolved via the Huber--Wüstholz analytic subgroup theorem, is recovered at depth 1 in our framework.
For 1-motives over number fields with good reduction at $p$, we identify canonical $\mathbb{Q}$-structures on the $p$-adic realizations arising from rational points of their associated formal $p$-divisible groups. Relative to these structures, we establish the conjectures at depths 1 and 2.

A key tool is the development of a $p$-adic analytic subgroup theorem tailored to 1-motives, providing an analogue of Wüstholz’s classical result. Our work not only yields a $p$-adic counterpart to the Kontsevich--Zagier conjecture for 1-motives but also opens new pathways for the study of linear relations among $p$-adic periods and their transcendence properties.

\end{abstract}
\maketitle


\section{Introduction}

The theory of periods lies at the confluence of algebraic geometry, transcendental number theory, and arithmetic geometry. Classically, periods are defined as complex numbers obtained by integrating algebraic differential forms over topological cycles on algebraic varieties. These numbers encode profound arithmetic and geometric information and have played a central role in major mathematical developments, such as the proof of the transcendence of $\pi$ and $e$, and the Lindemann--Weierstrass theorem. More broadly, Grothendieck's period conjecture posits that all algebraic relations among periods are governed by the structure of the underlying motive and its motivic Galois group, offering a tantalizing glimpse into the deep symmetries of arithmetic geometry. This idea was first introduced in \cite{grothendieck1966rham} and was subsequently formalized in full detail in \cite{And04} and \cite{And17}.

Despite significant advances, our understanding of periods remains incomplete. While transcendence and linear independence results are known for certain classes of periods, such as those associated to abelian varieties and some multiple zeta values, the general conjectures remain widely open. In this context, \emph{1-motives}, introduced by Deligne, provide a tractable yet rich class of mixed motives. They encompass extensions of tori and abelian varieties and serve as a natural bridge between pure and mixed motives. Remarkably, recent work by Huber and Wüstholz (\cite{huber2022transcendence}) has shown that all linear relations among the classical periods\footnote{We use the term "classical" to distinguish these periods from p-adic periods, which are the main focus of this paper.} of 1-motives over $\barQ$ are induced solely by bilinearity and functoriality, thereby verifying the Kontsevich--Zagier period conjecture for this class.

Parallel to these developments in the complex setting, $p$-adic Hodge theory has emerged as a powerful framework for understanding $p$-adic analogues of cohomological invariants. The comparison isomorphisms established by Fontaine, Faltings, and others reveal intricate connections between the étale, de Rham, and crystalline cohomologies of algebraic varieties over $p$-adic fields. These theories suggest the existence of a $p$-adic period theory mirroring the classical one. However, despite the foundational work of Fontaine, Colmez, and later developments by Faltings, Andr\'e, Furusho, and Ayoub, a direct $p$-adic period theory capable of addressing questions of linear relations and transcendence has not been fully realized.

Existing approaches to $p$-adic periods have primarily focused on constructing period matrices and comparison isomorphisms in cohomology. While these constructions provide a rich structural framework, they do not directly tackle questions concerning the linear relations among $p$-adic periods or their transcendence properties. Moreover, they are often restricted to pure motives or varieties with specific cohomological constraints, leaving mixed motives, and in particular 1-motives, less explored in the $p$-adic setting.

The purpose of this work is to develop a comprehensive $p$-adic theory of periods for 1-motives with good reduction. We extend the classical framework by constructing a p-adic integration pair-
ing for 1-motives, generalizing the celebrated theory of Colmez and Fontaine for abelian varieties.
For a 1-motive $M = [L \to G]$ over a $p$-adic local field $K$ with good reduction, we construct a bilinear, perfect, Galois-equivariant pairing
    \[
    \int^{\varpi}: T_p(M) \times T_\mathrm{dR}^\vee(M) \to B_2,
    \]
    which is compatible with Hodge filtration, where $B_2 = \mathbf{B}_\mathrm{dR}^{+}/t^2 \mathbf{B}_\mathrm{dR}^{+}$ is a natural quotient of de Rham period ring. This generalizes the $p$-adic integration theory of Colmez and Fontaine beyond abelian varieties (\cref{theorem p-adic integration pairing for M}).

Building on this construction, we introduce a stratified formalism for $p$-adic periods, inspired by the structure of period algebras in the classical setting. For each 1-motive, we define period spaces stratified by depth, where each level captures increasingly refined linear relations among periods. In particular, the validity of this $p$-adic period conjecture implies that all relations among the constructed $p$-adic periods are accounted for by bilinearity and functoriality. Moreover, increasing the depth reveals additional relations that go beyond those induced by bilinearity and functoriality. This extends the Kontsevich--Zagier philosophy to the $p$-adic realm.

A central innovation in our approach is the construction of canonical $\mathbb{Q}$-structures on the $p$-adic realizations of 1-motives with good reduction. These structures, denoted
\[
h_p(M), \quad H_p^\varphi(M), \quad \text{and} \quad H_p^\varpi(M),
\]
are derived from the $\mathbb{Q}$-rational points of the formal $p$-divisible group associated with $M$. They provide a natural framework for defining and analyzing the $p$-adic periods relative to arithmetic data. More precisely, we construct certain \(\Q\)-structures \(F(M) \subset \Tp(M) \otimes \Bt\) and linear subspaces \(G(M) \subset \TdRv(M) \otimes \overline{\Q}\), relative to which we precisely characterize all relations among the \(p\)-adic periods of \(M\) relative to $(F,G)$. Moreover, this allows us to determine exactly when there are relations beyond those induced by bilinearity and functoriality. Here, a $p$-adic period of $M$ relative to $(F,G)$ is a p-adic number which is in the image of 
    \[
    \left.\int^{\varpi}\right|_{F(M)\times G(M)}:F(M)\times G(M)\to\Bt.
    \]

The key technical tool underpinning our results is a $p$-adic analogue of the analytic subgroup theorem, specifically adapted to 1-motives (\cref{p-adic subgroup theorem for Fontaine pairing}). In the complex setting, Wüstholz's analytic subgroup theorem (\cite{wustholz1989algebraische}) has been instrumental in establishing transcendence and linear independence results for periods. Our $p$-adic version plays a similarly crucial role, enabling us to control the vanishing of $p$-adic periods and to prove the linear independence results predicted by our conjectures at various depths.

Roughly, we prove that for 1-motives over number fields with good reduction at a prime $p$, the $p$-adic period conjectures hold at depths 1 and 2 relative to the constructed canonical $\mathbb{Q}$-structures (\cref{level 2 period cinjecture for Hpp} and \cref{depth 2 and 1 period cinjecture for Hp and hp}). This provides a $p$-adic counterpart to the classical results for 1-motives, highlights the $\barQ$-linear independence of our p-adic periods (see, for instance, \cref{prop 4.4.1}) and represents a significant step toward a broader $p$-adic transcendence theory.

\medskip

Our results build upon and extend several important developments in $p$-adic Hodge theory and transcendence. They generalize the integration framework of Colmez and Fontaine for abelian varieties, adapt Wüstholz’s analytic subgroup theorem to the $p$-adic setting, and connect with recent work on period algebras and $p$-adic period conjectures for mixed motives. Unlike earlier approaches, which are largely structural and cohomological, ours is explicitly designed to detect all relations among the constructed $p$-adic periods, producing concrete period spaces and canonical $\mathbb{Q}$-structures, and and providing a framework for addressing the $p$-adic period conjectures. By employing $p$-adic integration and formal group techniques, we avoid conjectural identifications with crystalline cohomology classes and obtain unconditional results for 1-motives.

Significant progress toward $p$-adic period theory has been made in \cite{ancona2022algebraic}, \cite{Ayoub2020NouvellesCD}, \cite{brown_notes_2017}, and \cite{furusho_p_2004}. For a smooth projective variety $X$ over $\bar{\Q}$ with good reduction at $p$, Berthelot’s comparison theorem \cite[Theorem V.2.3.2]{Berthelot} gives a canonical isomorphism:
\begin{equation}\label{2}
    \operatorname{H}^n_{\mathrm{dR}}(X/\bar{\Q}) \otimes \bar{\Q}_p \cong \operatorname{H}^n_{\mathrm{crys}}(\bar{X}, \bar{\Q}_p).
\end{equation}
Using this, Ancona et al.\ define \emph{André’s $p$-adic periods} as the matrix entries of $cl_p(X) \hookrightarrow \operatorname{H}_{\mathrm{dR}}(X) \otimes \bar{\Q}_p$ with respect to $\bar{\Q}$-bases, where $cl_p(X)$ denotes algebraic classes modulo $p$ in crystalline cohomology. Whether these periods arise from Coleman--Berkovich--Vologodsky $p$-adic integration \cite{besser2000generalization, berkovich2007integration, vologodsky2001hodge} remains unresolved.

Their motivic framework yields an upper bound on the transcendence degree of the period algebra $\cP_p(\langle M \rangle)$ for a motive $M$ with good reduction:
\[
\dim \cP_p(\langle M \rangle) \leq \dim G_{\mathrm{dR}}(M) - \dim G_{\mathrm{crys}}(M),
\]
where $G_{\mathrm{dR}}(M)$ and $G_{\mathrm{crys}}(M)$ are the Tannakian Galois groups in the de Rham and crystalline settings. The $p$-adic Grothendieck period conjecture asserts equality, but remains open beyond the case of Kummer motives $M(\alpha)$ with $\alpha \in \Q$ \cite{ancona2022algebraic, Yamashita2010}.

While foundational, this approach offers no insight into linear relations or vanishing behaviour among $p$-adic periods, does not extend to curves (where algebraic classes in crystalline cohomology are inaccessible), and excludes 1-motives. It also depends on the standard conjecture $\sim_{\mathrm{num}} = \sim_{\mathrm{hom}}$. In contrast, our construction applies to 1-motives, yields all $p$-adic period relations explicitly, classifies all vanishing periods, and our framework avoids conjectural assumptions.

For mixed Tate motives, analogous constructions appear in \cite{furusho_p_2004, brown_notes_2017}, where Frobenius actions on crystalline cohomology yield $p$-adic periods via \eqref{2}. In this case, André’s periods come indeed from these Frobenius coefficients \cite{ancona2022algebraic}.

The abstract period algebra developed in \cite{Ayoub2020NouvellesCD} relates closely to André’s and fits within our formalism. Still, the relation between their periods and our $\hp$-, $\Hp$-, and $\Hpp$-periods remains unclear. Ours arise from explicit $p$-adic integration (Section~3), whereas it is unknown whether André’s do. Even for abelian varieties $A$, it is not known whether our $\mathbb{Q}$-structures $\hp(A)$, $\Hp(A)$, or $\Hpp(A)$ correspond to algebraic classes in $\operatorname{H}^2_{\mathrm{crys}}(\bar{A})$. We leave a systematic comparison to future work.

 The paper is organized as follows:

\begin{itemize}
    \item In Section 2, we review the theory of 1-motives, their cohomological realizations, and the relevant background in $p$-adic Hodge theory.
    \item In Section 3, we develop the $p$-adic integration theory for 1-motives with good reduction, establishing the existence and properties of the integration pairing.
    \item Section 4 introduces the stratified formalism of $p$-adic periods, defines canonical $\mathbb{Q}$-structures, and formulates the period conjectures
    \item In Section 5, we establish a $p$-adic version of the analytic subgroup theorem for $1$-motives and apply it to prove the main results concerning the $p$-adic period conjectures. These results involve periods rising from the pairings induced by $p$-adic integration on 1-motives and are formulated relative to the constructed $\Q$-structures.
    \item Section 6 provides explicit examples illustrating the theory, including 1-motives associated with smooth projective varieties and their $p$-adic periods.
\end{itemize}

We conclude with a discussion of open problems and potential extensions, including the study of $p$-adic period domains and the extension of our methods to more general mixed motives.

\subsection*{Conventions}
\begin{itemize}
    \item We fix a prime number $p$.
    \item Unless otherwise specified, $K$ is typically a non-archimedean local field of mixed characteristic $(0,p)$ with perfect residue field $k$ and valuation ring $\cO_K$. The algebraic closure of $K$ is denoted by $\bar{K}$, and $\C_p$ is the completion of $\bar{K}$. Let us denote by $\Gamma_K$ the absolute Galois group of $K$.
    \item We denote the maximal unramified extension of $K$ by $K^{ur}$, with its residue field denoted by $\bar{\F_p}$.
    \item $\K$ denotes a number field, and $\bar{\K}=\barQ$ is its algebraic closure. We denote by $\Gamma_{\K}$ the absolute Galois group of $\K$.
    \item The scheme $S$ is always a locally Noetherian, integral, regular scheme.
    \item By $S$-group schemes, we always mean the commutative ones.
    \item For any abelian category $\cC$ and any object $X\in\cC$, we denote by $\langle X\rangle$ the full abelian subcategory generated by $X$.
\end{itemize}

\tableofcontents


\section{1-motives and their realizations}

A \emph{1-motive} $M$ over a base scheme $S$ is a two-term complex of commutative $S$-group schemes
\[
M = [L \xrightarrow{u} G],
\]
where:
\begin{itemize}
    \item $L$ is a lattice over $S$, that is, an $S$-group scheme locally (for the étale topology) isomorphic to a constant finitely generated free $\mathbb{Z}$-module;
    \item $G$ is a semi-abelian $S$-scheme, i.e., an extension of an abelian $S$-scheme $A$ by a torus $T$ over $S$;
    \item $u$ is a morphism of $S$-group schemes.
\end{itemize}
Morphisms of 1-motives are morphisms of complexes, and we place $L$ in degree $-1$ and $G$ in degree $0$. The category of 1-motives over $S$ is denoted by $\Mi(S)$.

An important structural feature of a 1-motive is the canonical exact sequence:
\begin{equation}\label{canonical exact sequence for any 1-motive}
    0 \to [0 \to G] \to M \to [L \to 0] \to 0.
\end{equation}
This sequence expresses $M$ as an extension of a pure lattice by a semi-abelian scheme.
\begin{example}
Let $S = \Spec K$ for a field $K$. A 1-motive $M = [L \xrightarrow{u} G]$ consists of a finitely generated free $\mathbb{Z}$-module $L$ equipped with a continuous $\Gamma_K$-action, and a $\Gamma_K$-equivariant  homomorphism  $u: L \to G(\bar{K})$.
\end{example}

To work within an abelian framework, we pass to the isogeny category  $\mathcal{M}_1(S) \otimes \mathbb{Q}$ which we continue to denote by  $\mathcal{M}_1(S)$. From now on, we consider 1-motives up to isogeny.

\subsection{Weight Filtration}

1-motives naturally carry a canonical weight filtration, reflecting their
mixed nature. For $M = [L \to G]$, define:
\[
\W_i(M) =
\begin{cases}
    0 & i < -2, \\
    [0 \to T] & i = -2, \\
    [0 \to G] & i = -1, \\
    M & i \geq 0,
\end{cases}
\]
where $T$ is the toric part of $G$. The associated graded pieces are:
\[
\gr^W_i(M) =
\begin{cases}
    0 & i \leq -3 \text{ or } i \geq 1, \\
    [0 \to T] & i = -2, \\
    [0 \to A] & i = -1, \\
    [L \to 0] & i = 0,
\end{cases}
\]
with $A$ the abelian quotient of $G$. From now on, we view a torus \(T\), a lattice \(L\), an abelian scheme \(A\), and a semi-abelian scheme \(G\) as 1-motives, identified with the complexes \([0 \to T]\), \([L \to 0]\), \([0 \to A]\), and \([0 \to G]\), respectively. This filtration plays a central role in describing the realizations of 1-motives.
 
\subsection{Cartier Duality}

The duality theory for 1-motives mirrors that of abelian schemes, but with
added complexity due to the presence of tori and lattices. Consider the complex $M_A = [L \xrightarrow{v} A]$, where $v$ is the composition of $u$ with the natural projection $G \to A$. This fits into the commutative diagram
\[
\begin{tikzcd}
& & L \arrow[d, "u"] \arrow[rd, "v"] & \\
0 \arrow[r] & T \arrow[r] & G \arrow[r] & A \arrow[r] & 0.
\end{tikzcd}
\]
The Cartier dual $M^\vee$ of $M$ is the 1-motive 
\[M^\vee = [T^\vee \xrightarrow{u^\vee} G^\vee],\]
where:
\begin{itemize}
    \item $T^\vee=\underline{\operatorname{Hom}}_S(T, \mathbf{G}_m)$ is the character lattice of $T$;
    \item $L^\vee=\underline{\operatorname{Hom}}_S(L, \mathbf{G}_m)$ is the torus dual of $L$;
    \item $A^\vee$ is the dual abelian scheme of $A$, representing the sheaf $\operatorname{Ext}^1_S(A, \mathbf{G}_m)$ by the Weil-Barsotti formula;
    \item $G^\vee$ is the extension of $A^\vee$ by $L^\vee$, representing the sheaf $\operatorname{Ext}^1_S(M_A, \mathbf{G}_m)$.
\end{itemize}
This duality fits into the diagram:
\[
\begin{tikzcd}
& & T^\vee \arrow[d, "u^\vee"] \arrow[rd, "v^\vee"] & \\
0 \arrow[r] & L^\vee \arrow[r] & G^\vee \arrow[r] & A^\vee \arrow[r] & 0.
\end{tikzcd}
\]
Cartier duality defines an exact contravariant functor $(\cdot)^\vee : \mathcal{M}_1(S) \to \mathcal{M}_1(S)$, satisfying a canonical isomorphism
\[(M^\vee)^\vee \cong M.\]
\subsection{Realization Functors}

We now turn to the homological realizations of 1-motives, beginning with the de Rham realization via universal vector extensions.

A \emph{vector extension} of  $M = [L \to G]$ over $S$ is an extension of $M$ by a vector group $V$, i.e., an exact sequence of complexes
\[
0 \longrightarrow [0 \to V] \longrightarrow [L \to G'] \longrightarrow [L \to G] \longrightarrow 0.
\]
We write  $[0 \to V]$ simply as $V$.

Among all such extensions, the \emph{universal vector extension} $M^\natural$ of $M$ is characterized by the universal property:  for any vector extension as above,
there exists a unique homomorphism of $S$-vector groups $\varphi: V(M) \to V$ such that the extension $[L \to G']$ is the push-out of $M^\natural$ by $\varphi$.

The universal vector extension exists and fits into the exact sequence
\[
0 \longrightarrow V(M) \longrightarrow M^\natural \longrightarrow M \longrightarrow 0,
\]
where
$
V(M) = \underline{\operatorname{Hom}}_S\left( \operatorname{Ext}^1_S(M, \mathbf{G}_a), \mathcal{O}_S \right),
$ and $M^\natural = [L \xrightarrow{u^\natural} G^\natural]$\footnote{Over a field of characteristic zero, see \cite[10.1.7]{deligne1974theorie}. For general case over $S$, see \cite[Prop. 2.2.1]{barbieri2009sharp}, and \cite[\S 2.3, 2.4]{andreatta2005crystalline}}. Explicitly, we have an extension of $S$-group schemes
\[
0 \longrightarrow V(M) \longrightarrow G^\natural \longrightarrow G \longrightarrow 0.
\]
such that $G^\natural$ is the push-out of the universal vector extension
\[
0 \longrightarrow V(G) \longrightarrow E(G) \longrightarrow G \longrightarrow 0
\]
of the semi-abelian scheme $G$ along the natural inclusion $ V(G)=\ihom_{\cO_S}(\iext^1_S(G,\G_{a,S}),\cO_S)\hookrightarrow\ihom_{\cO_S}(\iext^1_S(M,\G_{a,S}),\cO_S)=V(M)$, yielding a non-canonical isomorphism
\[
G^\natural \simeq E(G) \times_S (L \otimes_{\mathbb{Z}} \mathbf{G}_a).
\]

The \emph{de Rham realization} of a 1-motive $M$ over $S$ is defined by:
\[
\TdR(M) := \underline{\Lie}_{G^\natural}(S) = \Lie(G^\natural),
\]
which is naturally endowed with a weight and a Hodge filtration:
\[
\Fil^i \TdR(M) =
\begin{cases}
    V(M) = \ker(\Lie(G^\natural) \to \Lie(G)) & i = 0, \\
    \TdR(M) & i = -1, \\
    0 & \text{otherwise}.
\end{cases}
\]
Thus, we obtain the canonical exact sequence:
\[
0 \to V(M) \to \TdR(M) \to \Lie(G) \to 0.
\]
By \cite[Prop.~2.3]{bertapelle_delignes_2009}, there is a canonical isomorphism:
\[
\underline{\Ext}^1_S(M, \mathbb{G}_{a,S}) \cong \underline{\Ext}^1_S(M_A, \mathbb{G}_{a,S}) \cong \underline{\Lie}(G^\vee),
\]
and hence:
\begin{equation}\label{identification V(M) and coLie}
V(M) \cong \underline{\Hom}_{\mathcal{O}_S}(\Lie(G^\vee), \mathcal{O}_S) = \underline{\coLie}(G^\vee).    
\end{equation}

In particular, there is a commutative diagram:
\[
\begin{tikzcd}
0 \arrow[r] & V(G) \arrow[r] \arrow[d, Rightarrow, no head] & V(M) \arrow[r] \arrow[d, Rightarrow, no head] & V(L) \arrow[r] \arrow[d, Rightarrow, no head] & 0 \\
0 \arrow[r] & \coLie(A^\vee) \arrow[r] & \coLie(G^\vee) \arrow[r] & \coLie(L^\vee) \arrow[r] & 0.
\end{tikzcd}
\]

\begin{prop}[{\cite[Lemma~2.3.2]{barbieri2009sharp}}]
The functor $M \mapsto M^\natural$ is exact.
\end{prop}

\subsection{$\ell$-adic and crystalline realizations}
The study of torsion points and their limits leads
to the $\ell$-adic and crystalline realizations of 1-motives.

Let $n$ be a positive integer. Consider the multiplication-by-$n$ map on $M$, denoted $n\colon M \to M$, acting by multiplication on both $L$ and $G$. It yields a commutative diagram:
\[
\xymatrix{
L \ar[r]^{u} \ar[d]^{n} & G \ar[d]^{n} \\
L \ar[r]^{u} & G
}
\]
which induces a morphism of $S$-group schemes
\[
L \to L \times_G G, \quad x \mapsto (nx, -u(x)).
\]
Define:
\[
M[n] := \coker\left(L \to L \times_G G\right).
\]
Equivalently, $M[n]$ is the cohomology sheaf $H^{-1}(M/n)$, where $M/n$ denotes the cone of multiplication by $n$ on $M$. The  exact sequence \eqref{canonical exact sequence for any 1-motive} gives rise to a short exact sequence:
\begin{equation}\label{exact sequence M[n]}
0 \to G[n] \to M[n] \to L[n] \to 0.
\end{equation}
Explicitly,
\begin{equation}\label{formula M[n]}
M[n] = \frac{\{(x, g) \in L \times G \mid u(x) = -ng\}}{\{(nx, -u(x)) \mid x \in L\}},
\end{equation}
a quotient in the fppf topology. In particular, if $S = \Spec K$, the group scheme $M[n]$ corresponds to a $(\Z/n\Z)[\Gamma_K]$-module.

Assume that $G$ is an extension of an abelian scheme $A$ by a torus $T$. It is well known that multiplication by $n$ in $A$ and $T$ is finite and faithfully flat, and that $T[n]$ and $A[n]$ are finite flat group schemes over $S$, which are, moreover, 'etale if $n$ is invertible on $S$. In this case, \eqref{exact sequence M[n]} implies that $M[n]$ is finite flat, and is \'etale whenever $S$ is defined over $\Z[1/n]$. 

Define the $p$-divisible group associated to $M$ by:
\[
M[p^{\infty}] := \varinjlim_n M[p^n],
\]
where the transition maps $M[p^m] \to M[p^n]$ for $m \geq n$ are given by $(x, g) \mapsto (p^{m-n}x, g)$. 

On the other hand, for a prime number $\ell$, The $\ell$-adic realization of a 1-motive $M$  is given by the Tate module 
\[T_\ell(M):=\lim_{m}M[\ell^m]\]
and the corresponding $\mathbb{Q}_\ell$-vector space:
\[
V_\ell(M) = T_\ell(M) \otimes_{\mathbb{Z}_\ell} \mathbb{Q}_\ell,
\]
where the inverse limit is taken over maps $M[\ell^m](\bar{K})\to M[\ell^n](\bar{K})$, for $m\geq n$, induced by $(x,g)\mapsto (x,\ell^{m-n}g)$.

If $M$ is a 1-motive over a $p$-adic local field $K$, we say that $M$ has \emph{good reduction} if it extends to $\cO_K$. Denote its reduction modulo $p$ by $\bar{M}$, and let  $\gMi(K)$ be the isogeny category of 1-motives with good reduction over $K$.

If \( K \) is a number field. We say that a 1-motive \( M \) has good reduction at a prime \( \mathfrak{p} \subset \mathcal{O}_K \) lying above \( p \) if the base change \( M \times_{\Spec K} \Spec K_{\mathfrak{p}} \) has good reduction, where \( K_{\mathfrak{p}} \) denotes the \( p \)-adic completion of \( K \) at \( \mathfrak{p} \).

\begin{remark}
As it is shown in \cite[Chapter IV]{MatevGoodReduction}, for a 1-motive $M$ over $K$, where $K$ is either a number field or a
p-adic local field, $M$ has a good reduction at $p$ if and only if $\Tl(M)$ is unramified at a prime $\ell$ which is different from $p$. As extensions of unramified representations are also unramified, we can conclude that the isogeny category $\gMi(K)$ of 1-motives over $K$ is abelian. 
\end{remark}

\begin{prop}\label{Hodge-Tate weights of tate module of 1-motives}
Let $M\in\gMi(K)$. Then the Tate module $\Vp(M)$ is a Hodge–Tate representation of $\Gamma_K$, with Hodge–Tate weights $0$ and $1$ of multiplicities $\rank(L) + \dim(A)$ and $\dim(T) + \dim(A)$, respectively.
\end{prop}

\begin{proof}
 Let $M$ be a 1-motive over $K$ with good reduction extending to $\mathcal{O}_K$. Consider the exact sequence of $p$-divisible groups over $\cO_K$:
 
\begin{equation}\label{exact sequence M[p^infty]}
0 \to G[p^{\infty}] \to M[p^{\infty}] \to L[p^{\infty}] \to 0.
\end{equation}

Since $L[p^{\infty}]$ is \'etale, the connected components of $M[p^{\infty}]$ and $G[p^{\infty}]$ coincide, and hence
\[
\Lie(M[p^{\infty}]) \cong \Lie(G[p^{\infty}]),\, \dim(\Lie(M[p^{\infty}]))=\dim(T)+\dim(A).
\]
Taking Cartier duals from sequence \ref{exact sequence M[p^infty]}, which preserves exactness for $p$-divisible groups, we obtain:
\[
0 \to (L[p^{\infty}])^{\vee} \to (M[p^{\infty}])^{\vee} \to (G[p^{\infty}])^{\vee} \to 0.
\]
As $L[p^{\infty}] \cong L \otimes \Q_p/\Z_p$, and since $\underline{\Hom}(L \otimes \mathbb{Q}_p/\mathbb{Z}_p, \mathbb{G}_m) \cong \mup^{\rank(L)}$, the dimension of $(L[p^{\infty}])^{\vee}$ equals $\rank(L)$.

To determine $\dim((G[p^{\infty}])^{\vee})$, observe that the sequence
\[
0 \to T[p^{\infty}] \to G[p^{\infty}] \to A[p^{\infty}] \to 0
\]
induces:
\[
0 \to (A[p^{\infty}])^{\vee} \to (G[p^{\infty}])^{\vee} \to (T[p^{\infty}])^{\vee} \to 0,
\]
with $(A[p^{\infty}])^{\vee} \cong A^{\vee}[p^{\infty}]$ (\cite[Theorem 19.1]{oort2006commutative}) and $(T[p^{\infty}])^{\vee}$ \'etale, hence $\dim((G[p^{\infty}])^{\vee}) = \dim(A\ve)=\dim(A)$.

Thus:
\[
\dim \Lie\left((M[p^{\infty}])^{\vee}\right) = \rank(L) + \dim(A).
\]
By Tate's theorem on Hodge–Tate decompositions for $p$-divisible groups (\cite[\S4]{tate1967pdivisble}), the Tate module $\Vp(G)$ of a $p$-divisible group $G$ over $\cO_K$ is a Hodge–Tate Galois representation, with Hodge-Tate weights $0$ and $1$ occurring with multiplicities equal to the dimensions of the Cartier dual $G\ve$ and of $G$, respectively. This concludes the proof.
\end{proof}

Let $k$ be a perfect field of characteristic $p$, $\W(k)$ its ring of Witt vectors equipped with the Frobenius automorphism $\sigma$, and $K_0$ the field of fractions of $\W(k)$. Let $\D$ denote the contravariant Dieudonn\'e crystal. Following \cite{andreatta2005crystalline}, the \emph{crystalline realization} of a 1-motive $M$ is the $\W(k)$-module:
\[
\Tcrys(M) := \varprojlim_n \D(M[p^{\infty}]^{\vee})(\Spec k \to \Spec \W_n(k)).
\]
Similarly, we define:
\[
\Tcrysv(M) := \varprojlim_n \D(M[p^{\infty}])(\Spec k \to \Spec \W_n(k)),
\]
called the \emph{Barsotti–Tate crystal} of $M$. They admit $\sigma$-semilinear operators: Frobenius $F$ and Verschiebung $V$ with $FV=VF=[p]$.

The functor $M \mapsto M[p^{\infty}]$ is exact and covariant, while the Dieudonn\'e functor and the Cartier dual are exact and contravariant. Hence, $M \mapsto \Tcrys(M)$ is exact and covariant.

For $M = [L \to G]$,  endow $\Tcrys(M)$ with the weight filtration:
\[
\W^i(\Tcrys(M)) =
\begin{cases}
\Tcrys(M), & i \geq 0, \\
\Tcrys(G), & i = -1, \\
\Tcrys(T), & i = -2, \\
0, & i \leq -3,
\end{cases}
\]
where $G$ is an extension of an abelian scheme $A$ by a torus $T$.

\begin{remark}\label{cor: Hodge filtration and tangent space on Dieudonne module}
Let $M = [L \to G]$ be a 1-motive over a $p$-adic local field $K$ with good reduction, and set $D = \Tcrys(\bar{M})$. By the crystalline--de Rham comparison isomorphism ( \cite[Theorem A']{andreatta2005crystalline}), we have a canonical isomorphism:
\begin{equation}\label{crystalline-de Rham comparison isomorphism}
D \otimes_{W(k)} K \cong \TdR(M_K).
\end{equation}
The Hodge filtration on $\TdR(M_K)$ induces a filtration on $N:=D \otimes_{\W(k)} K$ via the isomorphism above, endowing $N$ with the structure of a filtered isocrystal (or a filtered $\varphi$-modules); that is, an isocrystal $D\otimes_{\W(k)}K_0$ together with a structure of a filtered vector space over $K$ on $N$.
\end{remark}

This completes the construction of the standard homological realizations of 1-motives, providing the  foundation for the $p$-adic integration theory developed in the next section.


\section{P-adic integration theory for 1-motives with good reduction}\label{section: p-adic integration}

\subsection{Overview and Motivation}

Building on the foundational work of Coleman and Colmez, we aim to extend $p$-adic integration theory to the setting of 1-motives with good reduction. Coleman pioneered the theory of $p$-adic differential equations with applications to integration, while Colmez introduced a $p$-adic integration pairing for abelian varieties, which has since played a central role in $p$-adic Hodge theory and the study of $p$-adic periods.

Specifically, Colmez constructed a functorial, bilinear, and perfect pairing for an abelian variety $A$ with good reduction over a local field $K$:
\[
\int_A: \Tp(A) \times \operatorname{H}^1_{\mathrm{dR}}(A) \to \Bt,
\]
where $\Bt := \BdRp/I^2$ denotes the ring of truncated de Rham periods. Here,  $I$ is the maximal ideal of $\BdRp$, and $\C_p = \BdRp/I$ its residue field. 

This pairing is $\Gamma_K$-equivariant in the first argument and compatible with the Hodge filtration in the second. In particular, by restricting to the space of invariant differentials $\operatorname{H}^0(A, \Omega^1_A)$, Colmez obtains a refined pairing
\[
\langle\cdot, \cdot\rangle: \Tp(A) \times \operatorname{H}^0(A, \Omega^1_A) \to \C_p(1),
\]
recovering Fontaine's canonical integration map $\phi_A: \Tp(A) \to \Lie(A) \otimes_K \C_p(1)$.

The present section generalizes this construction to 1-motives $M = [L \to G]$ over a $p$-adic local field $K$ with good reduction. We define a $p$-adic integration pairing
\[
\int: \Tp(M) \times \TdRv(M) \to \Bt,
\]
which is functorial, bilinear, perfect, and compatible with both the Hodge and weight filtrations. It interpolates the pairings of Colmez and Fontaine in the cases $M = [0 \to A]$ and $M = [L \to 0]$, respectively.

We equip $\TdRv(M)$ with the Hodge filtration defined by
\begin{equation}
    \Fil^i\TdRv(M)=\begin{cases}
        \TdRv(M), & i \leq 0, \\
        \coLie(G), & i = 1, \\
        0, & i \geq 2.
    \end{cases}
\end{equation}
Let $t$ denote a uniformizer of $\BdRp$. The induced filtration on the period ring $\BdR = \BdRp[t^{-1}]$ is given by
\[
\Fil^i\BdR = t^i\BdRp / t^{i+1} \BdRp \cong \C_p(i), \quad \text{for } i \in \Z.
\]
This filtration induces a corresponding filtration on the quotient ring $\Bt = \BdRp / t^2 \BdRp$:
\[
\Fil^i \Bt = 
\begin{cases}
    \Bt, & i \leq 0, \\
    \C_p(1), & i = 1, \\
    0, & i \geq 2.
\end{cases}
\]

We say that the $p$-adic integration pairing
\[
\int \colon \Tp(M) \times \TdRv(M) \to \Bt
\]
respects the filtration if, for all $x \in \Tp(M)$ and all $i \in \Z$, we have
\[
\omega \in \Fil^i \TdRv(M) \quad \Rightarrow \quad \int_x \omega \in \Fil^i \Bt.
\]

This compatibility with the Hodge filtration reflects the deep analogy with classical period pairings, and will play a central role in formulating and proving $p$-adic period conjectures in the following sections.


\subsection{Fontaine’s Map  and Infinitesimal Deformations}

To build the integration theory for 1-motives, we begin by recalling the canonical construction of Fontaine’s map in the setting of
 semi-abelian schemes. This is achieved via deformation theory over the universal pro-infinitesimal
 thickening of $\cO_{\C_p}$.

Consider the natural inclusion $\BdRp \hookrightarrow \BdR$. It is known that the surjection $\theta\colon \BdRp \to \C_p$ admits a canonical section over $\bar{K} \subset \C_p$, i.e., there exists an embedding $\bar{K} \hookrightarrow \BdRp$ whose image is dense. Composing this embedding with the projection $\BdRp \to \Bt := \BdRp/I^2$ yields an injective map $\bar{K} \to \Bt$.

Let $\bar{I} := I/I^2$, and let $\Omega := \Omega_{\Spec(\cO_{\bar{K}})/\Spec(\cO_K)}$ denote the sheaf of relative Kähler differentials, with universal derivation $d\colon \cO_{\bar{K}} \to \Omega$. Define $\cA := \ker d$, so that we obtain the standard exact sequence
\[
0 \to \cA \to \cO_{\bar{K}} \xrightarrow{d} \Omega \to 0.
\]
Multiplication by $p^n$ induces a commutative diagram of $\cO_K$-modules:
\begin{equation}
\begin{tikzcd}
0 \arrow[r] & \cA \arrow[r] \arrow[d, "{[p^n]}", tail] & \cO_{\bar{K}} \arrow[r] \arrow[d, "{[p^n]}", tail] & \Omega \arrow[r] \arrow[d, "{[p^n]}", two heads] & 0 \\
0 \arrow[r] & \cA \arrow[r]                            & \cO_{\bar{K}} \arrow[r]                            & \Omega \arrow[r]                                 & 0.
\end{tikzcd}
\end{equation}
Since multiplication by $p^n$ is injective on $\cO_{\bar{K}}$ and surjective on $\Omega$, the snake lemma yields an exact sequence
\[
0 \to \Omega[p^n] \to \cA/p^n\cA \to \cO_{\bar{K}}/p^n\cO_{\bar{K}} \to 0.
\]
The system $\{\Omega[p^n]\}_{n \geq 1}$ forms an inverse system with surjective transition maps, hence satisfies the Mittag-Leffler condition. Passing to projective limits, we obtain the short exact sequence
\[
0 \to \Tp\Omega \to \widehat{\cA} \to \cO_{\C_p} \to 0,
\]
where $\widehat{\cA} := \varprojlim \cA/p^n\cA$. 

The ring $\BdRp$ is the universal pro-infinitesimal thickening of $\cO_{\C_p}$, and this universal property implies that there exists a unique isomorphism $\At \xrightarrow{\sim} \widehat{\cA}$ compatible with all structure maps where $\At$ is the ring of integers $\Bt$. In particular, we obtain a commutative diagram
\begin{equation}
\begin{tikzcd}
0 \arrow[r] & \bar{I} \arrow[r] \arrow[d, "\cong"] & \At \arrow[r] \arrow[d, "\cong"] & \cO_{\C_p} \arrow[r] \arrow[d, Rightarrow, no head] & 0 \\
0 \arrow[r] & \Tp\Omega \arrow[r]                  & \widehat{\cA} \arrow[r]          & \cO_{\C_p} \arrow[r]                                & 0,
\end{tikzcd}
\end{equation}
in which all maps are $\Gamma_K$-equivariant. Inverting $p$ in this diagram, we obtain isomorphisms
\[
\Tp\Omega[1/p] \cong t\Bt / t^2\Bt \cong \C_p(1), \qquad \Bt \cong \widehat{\cA}[1/p].
\]

This infinitesimal lifting property also has implications for group schemes. Namely, by the formal smoothness of $G$, the map
\[
G(\At) \to G(\cO_{\C_p}) = G(\At / \bar{I})
\]
is surjective, and its kernel is identified with $\Lie(G) \otimes_{\cO_K} \bar{I}$. Thus, we have the short exact sequence
\begin{equation} \label{exact seq for Fontaine'spairing}
\begin{tikzcd}
0 \arrow[r] & \Lie(G) \otimes_{\cO_K} \bar{I} \arrow[r] & G(\At) \arrow[r] & G(\cO_{\C_p}) \arrow[r] & 0.
\end{tikzcd}
\end{equation}

Multiplication by $p^n$ on $G$ gives rise to a commutative diagram:
\begin{equation}
\begin{tikzcd}
0 \arrow[r] & \Lie(G) \otimes_{\cO_K} \bar{I} \arrow[r] \arrow[d, "{[p^n]}"] & G(\At) \arrow[r] \arrow[d, "{[p^n]}"] & G(\cO_{\C_p}) \arrow[r] \arrow[d, "{[p^n]}"] & 0 \\
0 \arrow[r] & \Lie(G) \otimes_{\cO_K} \bar{I} \arrow[r]                      & G(\At) \arrow[r]                      & G(\cO_{\C_p}) \arrow[r]                       & 0.
\end{tikzcd}
\end{equation}
Applying the snake lemma again, we obtain a $\Gamma_K$-equivariant morphism
\[
\phi_n \colon G[p^n](\cO_{\C_p}) \to \Lie(G) \otimes_{\cO_K} \bar{I}/p^n\bar{I}.
\]
Using the natural identification $\Tp\Omega \cong \bar{I}$, we rewrite this as
\[
\phi_n \colon G[p^n](\cO_{\C_p}) \to \Lie(G) \otimes_{\cO_K} \Omega[p^n].
\]
Passing to the inverse limit and inverting $p$ yields the map
\[
\varphi_G \colon \Tp(G) \to \Lie(G) \otimes_K \C_p(1),
\]
which we call \emph{Fontaine’s map} for the semi-abelian scheme $G$. It is surjective after extension of scalars to $\C_p$.

The induced pairing 
\[
\Tp(G) \times \Lie^\vee(G) \to \C_p(1)
\]
 coincides with the Fontaine's pairing associated with the formal p
divisible group of G. 

\subsection{The $p$-adic Integration Pairing for Semi-Abelian Schemes}

We now refine the above construction using the universal vector extension of G, leading to a $p$-adic integration pairing that refines Fontaine’s.

 Let $\Ex(G)$ denotes the universal vector extension of $G$. This is a vector group extension 
\[
0 \to V \to \Ex(G) \to G \to 0,
\]
where $V = \Ext^1(G, \G_a)^\vee$.

We now construct the $p$-adic integration pairing associated with a semi-abelian scheme $G$ over $\cO_K$. The first step is to define the bilinear pairing
\[
\int^{\varpi}_G \colon \Tp(G) \times \coLie(\Ex(G)) \to \Bt,
\]
by analyzing the exact sequence of $\Bt$-points induced by lifting $G(\cO_{\C_p})$ to $\Ex(G)(\Bt)$.
The construction proceeds via the following diagram of sheaves with exact rows and columns:
\begin{equation} \label{big diagram for integration G}
\begin{tikzcd}
            & 0 \arrow[d]                                       & 0 \arrow[d]                                            & 0 \arrow[d]                                       &   \\
0 \arrow[r] & \Lie(V)\otimes_{\cO_K}\bar{I} \arrow[d] \arrow[r] & \Lie(\Ex(G))\otimes_{\cO_K}\bar{I} \arrow[d] \arrow[r] & \Lie(G)\otimes_{\cO_K}\bar{I} \arrow[r] \arrow[d] & 0 \\
0 \arrow[r] & V(\At) \arrow[d] \arrow[r]                        & \Ex(G)(\At) \arrow[d] \arrow[r]                        & G(\At) \arrow[r] \arrow[d]                        & 0 \\
0 \arrow[r] & V(\cO_{\C_p}) \arrow[d] \arrow[r]                 & \Ex(G)(\cO_{\C_p}) \arrow[d] \arrow[r]                 & G(\cO_{\C_p}) \arrow[r] \arrow[d]                 & 0 \\
            & 0                                                 & 0                                                      & 0                                                 &  
\end{tikzcd}
\end{equation}

From this diagram, we deduce that the map $\Ex(G)(\At) \to G(\cO_{\C_p})$ is surjective, and that its kernel is given by
\[
\cK := \frac{(V \otimes_{\cO_K} \At) \oplus (\Lie(\Ex(G)) \otimes_{\cO_K} \bar{I})}{\Lie(V) \otimes_{\cO_K} \bar{I}},
\]
where $\Lie(V) \cong V$ is embedded diagonally into the direct sum. This yields the short exact sequence
\begin{equation} \label{4.2.7}
\begin{tikzcd}
0 \arrow[r] & \cK \arrow[r] & \Ex(G)(\At) \arrow[r] & G(\cO_{\C_p}) \arrow[r] & 0.
\end{tikzcd}
\end{equation}

Next, we analyze the behavior under multiplication by $p^n$. This produces a commutative diagram:
\begin{equation}
\begin{tikzcd}
0 \arrow[r] & \cK \arrow[d, "{[p^n]}"] \arrow[r] & \Ex(G)(\At) \arrow[d, "{[p^n]}"] \arrow[r] & G(\cO_{\C_p}) \arrow[d, "{[p^n]}"] \arrow[r] & 0 \\
0 \arrow[r] & \cK \arrow[r]                      & \Ex(G)(\At) \arrow[r]                      & G(\cO_{\C_p}) \arrow[r]                      & 0
\end{tikzcd}
\end{equation}
Applying the snake lemma, we obtain a connecting map
\[
G[p^n](\cO_{\C_p}) \to \cK/p^n\cK.
\]
Composing with the natural projection
\[
\cK/p^n\cK \to \Lie(\Ex(G)) \otimes_{\cO_K} \At/p^n\At,
\]
we define a map
\[
\varpi_{n,G} \colon G[p^n](\cO_{\C_p}) \to \Lie(\Ex(G)) \otimes_{\cO_K} \At/p^n\At.
\]
Passing to the inverse limit over $n$, we obtain the desired map
\[
\varpi_G \colon \Tp(G) \to \TdR(G) \otimes_{\cO_K} \At,
\]
which we call the \emph{$p$-adic integration map}. This induces a bilinear pairing on the generic fibre
\[
\int^{\varpi}_G \colon \Tp(G_K) \times \TdRv(G_K) \to \Bt,
\]
where $G_K := G \times_{\Spec \cO_K} \Spec K$, and $\TdRv(G)$ denotes the linear dual of $\TdR(G)$.

\begin{prop} \label{theorem integration pairing for G}
The $p$-adic integration pairing $\int^{\varpi}$ is bilinear, perfect, and $\Gamma_K$-equivariant in the first argument. Moreover, it respects the Hodge filtrations.
\end{prop}

\begin{proof}
By construction, the pairing is bilinear and $\Gamma_K$-equivariant in the first argument. Let us verify the perfectness and filtration compatibility by considering several cases.

\smallskip

\textbf{Case 1: } $G = \G_m$. In this case, $V = 0$ since $\Ext^1(\G_m, \G_a) = 0$, and $\Ex(G) = \G_m$. Thus, we have
\[
\cK = \Lie(\G_m) \otimes_{\cO_K} \bar{I},
\]
and sequence \eqref{4.2.7} becomes
\[
0 \to \Lie(\G_m) \otimes_{\cO_K} \bar{I} \to \G_m(\At) \to \G_m(\cO_{\C_p}) \to 0.
\]
The resulting map $G[p^n](\cO_{\C_p}) \to \Lie(\G_m) \otimes_{\cO_K} \bar{I}/p^n\bar{I}$ lifts naturally to 
\[
\Lie(\G_m) \otimes_{\cO_K} \At/p^n\At,
\]
and hence taking inverse limits yields
\[
\Z_p(1) \to \Lie(\G_m) \otimes_K \C_p(1) \to \Lie(\G_m) \otimes_K \Bt.
\]
Thus, tensoring with $\Bt$ yields a $\Gamma_K$-equivariant isomorphism
\[
\Z_p(1) \otimes_{\Z_p} \Bt \xrightarrow{\sim} \Lie(\G_m) \otimes_{\cO_K} \Bt.
\]
The resulting pairing $\Z_p(1) \times \coLie(\G_m) \to \Bt$ respects the Hodge filtration, as the image of $\Z_p(1)$ lies in $\Lie(\G_m) \otimes_K \C_p(1) \subset \Lie(\G_m) \otimes_K \Bt$.

\smallskip

\textbf{Case 2: } $G = A$ an abelian scheme. In this case, the perfectness and Hodge compatibility of the pairing are proved by Colmez in \cite[Theorem 5.2]{colmez1992periodes}.

\smallskip

\textbf{Case 3: } $G$ a semi-abelian scheme. Let
\[
0 \to T \to G \to A \to 0
\]
be the standard decomposition of $G$ into a torus $T$ and abelian quotient $A$. We apply a d\'evissage argument with the commutative diagram with exact rows:
\begin{equation*}
\begin{tikzcd}
0 \arrow[r] & \Tp(T) \arrow[r] \arrow[d] & \Tp(G) \arrow[r] \arrow[d] & \Tp(A) \arrow[r] \arrow[d] & 0 \\
0 \arrow[r] & \TdR(T) \otimes \Bt \arrow[r] & \TdR(G) \otimes \Bt \arrow[r] & \TdR(A) \otimes \Bt \arrow[r] & 0
\end{tikzcd}
\end{equation*}
Since both the left and right vertical pairings are perfect and compatible with the Hodge filtration, it follows that the middle pairing inherits these properties.

The general statement on the filtration will be further elaborated in \cref{theorem p-adic integration pairing for M}.
\end{proof}

\subsection{Extension to 1-Motives}

We now extend the construction of the $p$-adic integration pairing to a general 1-motive $M = [L \xrightarrow{u} G]$. The universal vector extension $M^\natural = [L \xrightarrow{u^\natural} G^\natural]$ of $M$ fits into the exact sequence
\begin{equation} \label{exact sequence for colman pam for 1-motive}
0 \to V(M) \to G^\natural \to G \to 0.
\end{equation}
Following the strategy used in diagram \eqref{big diagram for integration G}, we obtain the following commutative diagram with exact rows and columns:
\begin{equation} \label{big diagram for integration M}
\begin{tikzcd}
            & 0 \arrow[d]                                       & 0 \arrow[d]                                            & 0 \arrow[d]                                       &   \\
0 \arrow[r] & \Lie(V(M))\otimes_{\cO_K}\bar{I} \arrow[d] \arrow[r] & \Lie(G^\natural)\otimes_{\cO_K}\bar{I} \arrow[d] \arrow[r] & \Lie(G)\otimes_{\cO_K}\bar{I} \arrow[r] \arrow[d] & 0 \\
0 \arrow[r] & V(M)(\At) \arrow[d] \arrow[r]                        & G^\natural(\At) \arrow[d] \arrow[r]                     & G(\At) \arrow[r] \arrow[d]                         & 0 \\
0 \arrow[r] & V(M)(\cO_{\C_p}) \arrow[d] \arrow[r]                 & G^\natural(\cO_{\C_p}) \arrow[d] \arrow[r]              & G(\cO_{\C_p}) \arrow[r] \arrow[d]                  & 0 \\
            & 0                                                 & 0                                                      & 0                                                 &  
\end{tikzcd}
\end{equation}

This diagram shows that the map $G^\natural(\At) \to G(\cO_{\C_p})$ is surjective with kernel
\[
\cK := \frac{(V(M) \otimes_{\cO_K} \At) \oplus (\Lie(G^\natural) \otimes_{\cO_K} \bar{I})}{\Lie(V(M)) \otimes_{\cO_K} \bar{I}}.
\]
We thus obtain the short exact sequence
\begin{equation} \label{7.3.2}
\begin{tikzcd}
0 \arrow[r] & \cK \arrow[r] & G^\natural(\At) \arrow[r] & G(\cO_{\C_p}) \arrow[r] & 0.
\end{tikzcd}
\end{equation}

To proceed, we define maps
\begin{align*}
q \colon L \times_{\cO_K} G &\to G, \quad (x,g) \mapsto u(x) + p^n g, \\
q^\natural \colon L \times_{\cO_K} G^\natural &\to G^\natural, \quad (x,g) \mapsto u^\natural(x) + p^n g.
\end{align*}
We denote by $\widetilde{M[p^n]} := \ker q$. Then, by \cref{formula M[n]}, the group scheme $M[p^n]$ is obtained as the quotient of $\widetilde{M[p^n]}$ by the image of $L$.

These maps and the exact sequence \eqref{7.3.2} yield the following commutative diagram with exact rows:
\begin{equation} \label{4.3.4}
\begin{tikzcd}
0 \arrow[r] & \cK \arrow[r] \arrow[d, "{[p^n]}"] & L \times_{\cO_K} G^\natural(\At) \arrow[r] \arrow[d, "q^\natural"] & L \times_{\cO_K} G(\cO_{\C_p}) \arrow[r] \arrow[d, "q"] & 0 \\
0 \arrow[r] & \cK \arrow[r]                         & G^\natural(\At) \arrow[r]                                   & G(\cO_{\C_p}) \arrow[r]                               & 0.
\end{tikzcd}
\end{equation}

Applying the snake lemma gives a map
\[
\widetilde{M[p^n]} \to \cK/p^n\cK,
\]
which, upon composition with the projection $\cK/p^n\cK \to \Lie(G^\natural) \otimes_{\cO_K} \At/p^n\At$, yields a map
\[
\widetilde{M[p^n]} \to \Lie(G^\natural) \otimes_{\cO_K} \At/p^n\At.
\]
Since the image of $L$ in $\widetilde{M[p^n]}$ maps to zero, the morphism factors through the quotient, inducing a well-defined map
\[
\varpi_{n,M} \colon M[p^n] \to \TdR(M) \otimes_{\cO_K} \At/p^n\At.
\]
Passing to the inverse limit over $n$ gives what we will call the $p$-adic integration map for the motive $M$:
\begin{equation} \label{Integration map for M}
\varpi_M \colon \Tp(M) \to \TdR(M) \otimes_{\cO_K} \At.
\end{equation}

This induces the $p$-adic integration pairing:
\begin{equation} \label{integration pairing for M}
\int^\varpi \colon \Tp(M_K) \times \TdRv(M_K) \to \Bt,
\end{equation}
where $M_K := M \times_{\Spec \cO_K} \Spec K$ is the generic fibre.

\smallskip

\noindent Similarly, by applying the same construction to diagram \eqref{exact seq for Fontaine'spairing}, we obtain the diagram
\begin{equation} \label{equation 4.3.6}
\begin{tikzcd}
0 \arrow[r] & \Lie(G)\otimes_{\cO_K}\bar{I} \arrow[r] \arrow[d, "{[p^n]}"] & L \times_{\cO_K} G(\At) \arrow[r] \arrow[d, "q_2"] & L \times_{\cO_K} G(\cO_{\C_p}) \arrow[r] \arrow[d, "q_1"] & 0 \\
0 \arrow[r] & \Lie(G)\otimes_{\cO_K}\bar{I} \arrow[r]                         & G(\At) \arrow[r]                                   & G(\cO_{\C_p}) \arrow[r]                               & 0,
\end{tikzcd}
\end{equation}
where $q_1$ and $q_2$ are induced by $(x,g) \mapsto u(x) + p^n g$. This yields a $\Gamma_K$-equivariant map
\[
\phi_{n,M} \colon M[p^n](\cO_{\C_p}) \to \Lie(G) \otimes_{\cO_K} \bar{I}/p^n\bar{I},
\]
and passing to the limit gives an analogue of Fontaine’s map for 1-motives:
\begin{equation}\label{Fontaine's map for M}
\varphi_M \colon \Tp(M) \to \Lie(G) \otimes_K \C_p(1).    
\end{equation}

The induced pairing
\begin{equation} \label{Fontaine's pairing for M}
\int^\varphi \colon \Tp(M) \times \coLie(G_K) \to \C_p(1)
\end{equation}
is referred to as Fontaine’s pairing for the motive $M$.

\smallskip

Finally, repeating the argument for the exact sequence \eqref{exact sequence for colman pam for 1-motive} leads to the diagram
\begin{equation} \label{4.3.9}
\begin{tikzcd}
0 \arrow[r] & V(M) \arrow[r] \arrow[d, "{[p^n]}"] & L \times G^\natural(\cO_{\C_p}) \arrow[r] \arrow[d, "q_2"] & L \times G(\cO_{\C_p}) \arrow[r] \arrow[d, "q_1"] & 0 \\
0 \arrow[r] & V(M) \arrow[r]  & L \times G^\natural(\cO_{\C_p}) \arrow[r] & L \times G(\cO_{\C_p}) \arrow[r] & 0,
\end{tikzcd}
\end{equation}
from which we obtain the map
\[
\psi_M \colon \Tp(M) \to V(M) \otimes_{\cO_K} \C_p,
\]
which we refer to as Coleman’s map for the 1-motive $M$.

\begin{remark}
    The Fontaine's map \eqref{Fontaine's map for M} generalizes the integration map $\phi_A: \Tp(A)\to \Lie(A)\otimes_K\C_p(1)$ for abelian variety $A$, originally introduced by Fontaine in \cite{Fontaine1982}. This map realizes Hodge-Tate decomposition after tensoring with $\C_p$. The same property holds for Fontaine's map $\phi_M$ associated to a 1-motive $M$. In particular, the map \[\phi_M\otimes 1_{\C_p}:\Tp(M)\otimes_{\Z_p}\C_p\to \Lie(G)\otimes_K\C_p(1)\] is surjective.
\end{remark}

\smallskip

We now prove the central result of this section. It generalizes \cite[Theorem 5.2]{colmez1992periodes} to 1-motives.

\begin{thm} \label{theorem p-adic integration pairing for M}
The $p$-adic integration pairing $\int^\varpi$ for a 1-motive $M$, given by \eqref{integration pairing for M}, is bilinear, perfect, and $\Gamma_K$-equivariant in the first argument. Moreover, it respects the Hodge filtration: for all $\omega \in \Fil^1 \TdRv(M_K)$ and $x \in \Tp(M)$, we have
\[
\int^\varpi_x \omega \,\in \Fil^1 \Bt.
\]
In particular, for $\omega \in \coLie(G_K) = \Fil^1 \TdRv(M_K)$, we have
\[
\int^\varpi_x \omega = \int^\varphi_x \omega.
\]
\end{thm}
\begin{proof}
The bilinearity and $\Gamma_K$-equivariance of the pairing follow directly from the construction.

We first verify perfectness in several cases. When $M = [L \to 0]$, we have $G^\natural = V(L)$ and
\[
\cK = V(L) \otimes_{\cO_K} \At / \bar{I} = V(L) \otimes \C_p.
\]
In this case, diagrams \eqref{4.3.4} and \eqref{4.3.9} coincide, and diagram \eqref{equation 4.3.6} becomes trivial. Thus, the $p$-adic integration map $\varpi_L$ agrees with Coleman’s map
\[
\psi_L \colon \Z_p \otimes L \to V(L) \otimes_{\cO_K} \C_p,
\]
given by $(x_n) \mapsto [p^n] x_n$ for $x_n \in L[p^n]$. The induced map
\[
L \otimes \Bt \to V(L) \otimes \Bt
\]
is clearly an isomorphism and respects the filtration.

For $M = [0 \to G]$, the pairing $\int^\varpi$ is perfect by Proposition \ref{theorem integration pairing for G}. In the general case $M = [L \to G]$, the canonical exact sequence of 1-motives yields a diagram with exact rows:
\[
\begin{tikzcd}[column sep=small]
0 \arrow[r] & \Tp(G) \arrow[r] \arrow[d, "\varpi_G"] & \Tp(M) \arrow[r] \arrow[d, "\varpi_M"] & \Z_p \otimes L \arrow[r] \arrow[d] & 0 \\
0 \arrow[r] & \Lie(\Ex(G)) \otimes \Bt \arrow[r] & \Lie(G^\natural) \otimes \Bt \arrow[r] & V(L) \otimes \Bt \arrow[r] & 0.
\end{tikzcd}
\]
Since the pairings induced by the left and right vertical arrows are perfect, the same holds for the middle one.

It remains to prove compatibility with the Hodge filtration. Since $V(M)$ is the kernel of the projection $\pi \colon G^\natural \to G$, the map $\Lie(G^\natural) \otimes \bar{I} \to \Lie(G) \otimes \bar{I}$ factors through $\cK$, yielding a map $g \colon \cK \to \Lie(G) \otimes \bar{I}$.

The commutative diagram
\[
\begin{tikzcd}
0 \arrow[r] & \cK \arrow[d, "g"] \arrow[r] & L \times G^\natural(\At) \arrow[d, "f"] \arrow[r] & L \times G(\C_p) \arrow[d, Rightarrow, no head] \arrow[r] & 0 \\
0 \arrow[r] & \Lie(G) \otimes \bar{I} \arrow[r] & L \times G(\At) \arrow[r] & L \times G(\C_p) \arrow[r] & 0
\end{tikzcd}
\]
along with the relation $q = q_1$, $\pi \circ q^\natural = q_2 \circ f$, and $g \circ [p^n] = [p^n] \circ g$ ensures that applying the snake lemma to diagrams \eqref{4.3.4} and \eqref{equation 4.3.6} yields a commutative square
\[
\begin{tikzcd}
M[p^n] \arrow[r] & \cK/p^n\cK \arrow[d] \\
M[p^n] \arrow[u, Rightarrow, no head] \arrow[r] & \Lie(G) \otimes \bar{I}/p^n\bar{I},
\end{tikzcd}
\]
which extends, after base change, to a commutative diagram
\[
\begin{tikzcd}
M[p^n] \arrow[r] & \cK/p^n\cK \arrow[d] \arrow[r] & \Lie(G^\natural) \otimes \Bt/p^n\Bt \arrow[d] \\
M[p^n] \arrow[u, Rightarrow, no head] \arrow[r] & \Lie(G) \otimes \bar{I}/p^n\bar{I} \arrow[r] & \Lie(G) \otimes \Bt/p^n\Bt.
\end{tikzcd}
\]
Therefore, the $p$-adic integration map $\varpi_M$ factors through Fontaine’s map $\varphi_M$:
\begin{equation}\label{commutative diagram respect filtration}
\begin{tikzcd}
\Tp(M) \arrow[r, "\varpi_M"] \arrow[d, "\varphi_M"] & \Lie(G^\natural) \otimes \Bt \arrow[d, two heads] \\
\Lie(G) \otimes \C_p(1) \arrow[r, hook] & \Lie(G) \otimes \Bt.
\end{tikzcd}
\end{equation}
This implies that $\int^\varpi_x \omega = \int^\varphi_x \omega$ whenever $\omega \in \Fil^1 \TdRv(M_K) = \coLie(G_K)$, completing the proof.
\end{proof}

\begin{remark}
Our construction of the Fontaine map $\varphi_M$ for 1-motives with good reduction not only extends Fontaine’s original map for abelian varieties, but also  enables
us to derive a statement which generalizes \cite[Theorem A.4]{iovita_p_2022} for the case of abelian varieties to 1-motives. This theorem sates that the kernel of Fontaine’s map consists of all $\Gamma_F$-invariant elements in the Tate module, where $F := K^{un}$ is the maximal unramified extension of $K$.
Moreover, we can obtain a criterion for injectivity of the $p$-adic integration map $\varpi_M$:
\end{remark}

\begin{prop} \label{kernel p-adic integration map varpi}
Let $M = [L \to G]$ be a 1-motive over $\cO_K$, and let $\sG^0$ denote the connected component of $G[p^\infty]$. If $\Tp(M)^{\Gamma_K} = 0$ and $\ker(\varpi_G) \cap \Tp(\sG^0) = 0$, then $\ker(\varpi_M) = 0$. In other words, under the vanishing of $\Gamma_K$-invariants in $\Tp(M)$, injectivity of $\left.\varpi_G\right|_{\Tp(\sG^0)}$ implies injectivity of $\varpi_M$.
\end{prop}

\begin{proof}
The canonical exact sequence
\[
0 \to \Tp(G) \to \Tp(M) \xrightarrow{f} \Tp(L) \to 0
\]
yields the long exact sequence in Galois cohomology:
\[
0 \to \Tp(G)^{\Gamma_K} \to \Tp(M)^{\Gamma_K} \to \Tp(L)^{\Gamma_K} \xrightarrow{\delta} H^1(K, \Tp(G)).
\]
Since $\Tp(G)^{\Gamma_K} = \Tp(M)^{\Gamma_K} = 0$, and $L[p^\infty]$ is étale, we deduce that $\Tp(L)^{\Gamma_K} = \Tp(L)$ and that the connecting map $\delta \colon \Tp(L) \to H^1(K, \Tp(G))$ is injective.

Let $x \in \Tp(M)$ with $\varpi_M(x) = 0$. Then for every $\sigma \in \Gamma_K$, we have $\varpi_M((\sigma - 1)x) = 0$. However, $(\sigma-1)x$ lies in the Tate module $\Tp(\sG^0)$, where $\sG^0$ is the connected component of $G[p^{\infty}]$ which is the same as the connected component of $M[p^{\infty}]$. Therefore, $(\sigma-1)x\in\ker(\varpi_G)\cap\Tp(\sG^0)$ for all $\sigma\in\Gamma_K$ and the assumption implies $(\sigma - 1)x = 0$. As $\sigma(x)=x$ for all $\sigma\in\Gamma_K$ and $f$ is $\Gamma_K$-equivariant, we can conclude that $\delta(f(x))=0$ in $H^1(K,\Tp(G))$. Since $\delta$ is injective, it follows that $f(x)=0$, which implies that $x\in\Tp(G)\cap\ker(\varpi_M)=\ker(\varpi_G)$.

Consider the exact sequence
    \begin{equation}\label{3.3.10}
    0\to \Tp(\sG^0)\to \Tp(G)\to \Tp(G)/\Tp(\sG^0)\to 0.
    \end{equation}
    The $\Z_p$-module $\Tp(G)/\Tp(\sG^0)$ is the Tate-module of the \'etale part of $G[p^{\infty}]$. If we repeat the same approach as above for the exact sequence \ref{3.3.10}, we obtain that $x\in\Tp(\sG^0)\cap\ker(\varpi_G)=0$. This shows that $\varpi_M$ is injective.
\end{proof}

\subsection{Crystalline integration}
The $p$-adic Tate module $\Tp(M)$ of a 1-motive $M$ with good reduction arises from the Tate module of a $p$-divisible group defined over $\cO_K$. Since the Tate module of a $p$-divisible group over $\cO_K$ is known to be crystalline in the sense of Fontaine, it follows that $\Tp(M)$ is a crystalline Galois representation. As a result, one obtains a canonical filtered isomorphism
\[
\Tp(M)\otimes_{\Z_p}\Bcrisp \cong \Tcrys(\bar{M})\otimes_{\W(k)}\Bcrisp,
\]
where the filtration on the right-hand side arises from the filtered isocrystal induced by the Hodge filtration, while the filtration on the left is induced by that of $\Bcrisp$ (see \cite{fontaine1994semi} for details). Here, $\Bcrisp$ is the positive crystalline period ring, equipped with a $\Gamma_K$-equivariant Frobenius $\phi_{\mathrm{cris}}$, naturally extended from the Frobenius automorphism on infinitesimal period ring $\Ainf$ (see \cite{fontaine1994corps} or \cite{colmez_construction_2000}). This identification gives rise to what we call the \emph{crystalline integration pairing}
\begin{equation}\label{crystalline integration map}
\int^{\mathrm{cris}} \colon \Tp(M)\times \Tcrysv(\bar{M})_K \to \Bcrisp.
\end{equation}
%
%
To further analyze this pairing, let $\sG$ denote the $p$-divisible group associated to $M$, and consider the associated connected component $\sG^0$. We define:
\begin{align}\label{tilde Tp and Tcrys}
\tilde{\Tp}(M) &:= \Tp(\sG^0), \\
D &:= \D(\sG^0),
\end{align}
where $D$ is the Dieudonné crystal associated to $\sG^0$. Note that $D$ corresponds to the submodule of $\Tcrysv(\bar{M})$ supported on non-zero slopes.

\begin{prop}\label{crystalline pairing and Frobenius}
For all $x \in \tilde{\Tp}(M)$, $\omega \in D$, and $n \geq 0$, we have
\[
\int^{\mathrm{cris}}_x F^n\omega = \phi_{\mathrm{cris}}^n\left(\int^{\mathrm{cris}}_x \omega\right),
\]
where $\phi_{\mathrm{cris}}$ denotes the Frobenius endomorphism on $\Bcrisp$.
\end{prop}
\begin{proof}
This is a direct consequence of \cite[Proposition~3.1]{colmez1992periodes}.
\end{proof}

We next establish a key comparison between Galois-invariants over $\Bt$ and $\BdRp$.

\begin{lemma}\label{lemma 4.3.1}
Let $V := \Vp(M)$. Then:
\[
(V \otimes \Bt)^{\Gamma_K} \cong (V \otimes \BdRp)^{\Gamma_K}, \quad \text{and} \quad (V^\vee \otimes \Bt)^{\Gamma_K} \cong (V^\vee \otimes \BdRp)^{\Gamma_K} \cong \DdR(V^\vee).
\]
\end{lemma}
\begin{proof}
We first prove that the canonical surjection $\BdRp \twoheadrightarrow \Bt$ induces an isomorphism on Galois invariants:
\[
(V \otimes_{\Q_p} \BdRp)^{\Gamma_K} \cong (V \otimes_{\Q_p} \Bt)^{\Gamma_K}.
\]
As $t\BdRp/(t^2\BdRp)\cong\C_p(1)$, we have an exact sequence
\[
0\to\C_p(1)\to\BdRp/t^2\BdRp\to \BdRp/t\BdRp\to 0
\]
of $K$-Banach spaces which, after tensoring with $V$, yields an exact sequence of $\Gamma_K$-modules
\[
0 \to V \otimes_{\Q_p} \C_p(1) \to V \otimes_{\Q_p} (\BdRp / t^2\BdRp) \to V \otimes_{\Q_p} (\BdRp / t\BdRp) \to 0.
\]
Since the Hodge--Tate weights of $V$ are $0$ and $1$, the Galois cohomology of $V \otimes \C_p(1) \cong \C_p(1)^r \oplus \C_p(2)^m$ vanishes in all degrees (\cite[Theorem~2]{tate1967pdivisble}). Taking the long exact sequence of Galois cohomology yields an isomorphism of 
Galois invariants:
\[
(V \otimes \Bt)^{\Gamma_K} \cong (V \otimes (\BdRp / t\BdRp))^{\Gamma_K}.
\]
Repeating this process inductively shows that for all $n \geq 1$, we have
\[
(V \otimes (\BdRp / t^n\BdRp))^{\Gamma_K} \cong (V \otimes \Bt)^{\Gamma_K}.
\]
Passing to the limit then gives
\begin{equation}\label{4.3.14}
(V \otimes \BdRp)^{\Gamma_K} \cong (V \otimes \Bt)^{\Gamma_K}.
\end{equation}

A similar argument applies to $V^\vee$, using the exact sequence
\[
0 \to V^\vee \otimes \C_p(n) \to V^\vee \otimes (\BdRp / t^{n+1}\BdRp) \to V^\vee \otimes (\BdRp / t^n\BdRp) \to 0
\]
for $n \geq 2$, along with the vanishing of Galois cohomology for $V^\vee \otimes \C_p(n) \cong \C_p(n)^r \oplus \C_p(n-1)^m$. It follows that
\[
(V^\vee \otimes \Bt)^{\Gamma_K} \cong (V^\vee \otimes \BdRp)^{\Gamma_K}.
\]

To identify this with $\DdR(V^\vee)$, consider the exact sequence
\[
0 \to \BdRp \to t^{-1}\BdRp \to \C_p(-1) \to 0,
\]
from which we get the following exact sequence of $\Z[\Gamma_K]$-modules
\[
0\to V\ve\otimes_{\Q_p}\BdRp\to V\ve\otimes_{\Q_p}t^{-1}\BdRp\to V\ve\otimes_{\Q_p}\C_p(-1)\to 0.
\]
The Hodge-Tate weights of $\Vp(M)$ are $0$ and $1$, therefore, the Hodge-Tate weights of $V\ve$ are $0$ and $-1$, and so
\[
V\ve\otimes_{\Q_p}\C_p(-1)\cong\C_p(-1)^r\oplus\C_p(-2)^m
\]
for some positive integers $r$ and $m$. As the Galois cohomology of $\C_p(-1)^r\oplus\C_p(-2)^m$ vanishes in all degrees, the exact sequence gives an isomorphism 
\begin{equation*}
(V\ve\otimes_{\Q_p}\BdRp)^{\Gamma_K}\cong (V\ve\otimes_{\Q_p}t^{-1}\BdRp)^{\Gamma_K}.    
\end{equation*}

Induction over $n \geq 2$ shows
\[
(V^\vee \otimes t^{-n}\BdRp)^{\Gamma_K} \cong (V^\vee \otimes t^{-n+1}\BdRp)^{\Gamma_K}.
\]
Using $\BdR=\colim t^{-n}\BdRp$ and taking the colimit gives
\begin{equation}\label{4.3.13}
\DdR(V^\vee) \cong (V^\vee \otimes \BdRp)^{\Gamma_K}.
\end{equation}
\end{proof}

\begin{cor}\label{tate module of M is de Rham}
Let $M$ be a 1-motive over $K$ with good reduction. Then there is a canonical isomorphism of filtered $K$-vector spaces:
\[
\DdR(\Tp M) \cong \TdR(M),
\]
and the $p$-adic Galois representation $\Vp(M)$ is de Rham.
\end{cor}
\begin{proof}
From Lemma~\ref{lemma 4.3.1} and Theorem~\ref{theorem p-adic integration pairing for M}, we obtain:
\[
\TdR(M) \cong (\Vp(M) \otimes \Bt)^{\Gamma_K} \cong (\Vp(M) \otimes \BdRp)^{\Gamma_K} \hookrightarrow \DdR(\Vp(M)).
\]
Since $\dim_K \TdR(M_K) = \dim_{\Q_p} \Vp(M)$ and $\dim \DdR(\Vp(M)) \leq \dim \Vp(M)$, the above embedding must be an isomorphism.
\end{proof}

The following corollary provides a crucial link between the two integration pairings.

\begin{cor}\label{cor 4.3.2}
With the notation of \eqref{crystalline integration map}, for all $x \in \Vp(M)$ and $\omega \in \TdR^\vee(M)_K$, we have
\[
\int^{\mathrm{cris}}_x \omega = 0 \quad \Longleftrightarrow \quad \int^{\varpi}_x \omega = 0.
\]
\end{cor}
\begin{proof}
Since $\Bcrisp \subset \BdRp$ and their filtrations are compatible, let $\omega' \in \Tcrysv(\bar{M})_K$ correspond to $\omega$ under the comparison isomorphism $\TdRv(M)_K\cong\Tcrysv(\bar{M})_K$. Consider the maps:
\[
f(x) := \int^{\varpi}_x \omega \in \Bt, \quad
g(x) := \int^{\mathrm{cris}}_x \omega' \in \Bcris \hookrightarrow \BdRp.
\]
Then
\[
f \in \Hom_{\Z_p[\Gamma_K]}(\Vp(M), \Bt) \cong (\Vp^\vee(M) \otimes \Bt)^{\Gamma_K}, \quad
g \in \Hom_{\Z_p[\Gamma_K]}(\Vp(M), \BdRp) \cong (\Vp^\vee(M) \otimes \BdRp)^{\Gamma_K}.
\]
By Lemma~\ref{lemma 4.3.1}, these spaces are naturally isomorphic. Thus, $f(x) = 0$ if and only if $g(x) = 0$.
\end{proof}


\section{P-adic peirods of 1-motives}

We develop a formalism that allows for a unified treatment of period pairings, encompassing both classical and $p$-adic settings. This framework provides a conceptual structure in which to study the linear relations among periods, including those arising from $p$-adic integration pairings. It also permits a precise formulation of period conjectures that quantify the extent to which all such relations are dictated by functoriality and bilinearity.

\subsection{Formalism of periods}
\begin{defn}[Realization category]
Let $K$ and $L$ be fields containing $\Q$ and $\bar{\Q}$, respectively, and let $B$ be an algebra over both $K$ and $L$. The realization category with coefficients in $B$, denoted $\Mod^{B}_{K,L}$, is defined as follows:
\begin{enumerate}
    \item[(i)] Objects are triples $(H_K,H_L,\varpi)$, where $H_K$ and $H_L$ are finite-dimensional vector spaces over $K$ and $L$, respectively, and $\varpi\colon H_K \otimes_K B \to H_L \otimes_L B$ is a $B$-linear isomorphism. This map encodes a comparison between two realizations via the base algebra $B$.
    \item[(ii)] Morphisms are pairs $\varphi := (\varphi_K, \varphi_L)$, where $\varphi_K\colon H_K \to H'_K$ is $K$-linear, $\varphi_L\colon H_L \to H'_L$ is $L$-linear, and the diagram
    \[
    \begin{tikzcd}
    H_K \otimes_K B \arrow[r, "\varpi"] \arrow[d, "\varphi_K \otimes 1_B"] & H_L \otimes_L B \arrow[d, "\varphi_L \otimes 1_B"] \\
    H'_K \otimes_K B \arrow[r, "\varpi'"] & H'_L \otimes_L B
    \end{tikzcd}
    \]
    commutes. This ensures compatibility of morphisms with the comparison isomorphisms.
\end{enumerate}
\end{defn}

The category $\Mod^{B}_{K,L}$ is a rigid abelian tensor category, equipped with the following structure:
\begin{itemize}
    \item \textbf{Tensor product:} Given objects $H = (H_K, H_L, \varpi)$ and $H' = (H'_K, H'_L, \varpi')$, the tensor product is defined by
    \[
    H \otimes H' = (H_K \otimes_K H'_K, H_L \otimes_L H'_L, \varpi \otimes \varpi').
    \]
    This defines a natural monoidal structure, compatible with the comparison morphisms.

    \item \textbf{Duals:} The dual of $H$ is $H^{\vee} = (\Hom(H_K, K), \Hom(H_L, L), \varpi^{\vee})$, where $\varpi^{\vee}(f) = f \circ \varpi^{-1}$ under the canonical identification $\Hom(H_K \otimes_K B, B) \cong \Hom(H_K, K) \otimes_K B$. This duality turns $\Mod^{B}_{K,L}$ into a rigid tensor category.

    \item \textbf{Unit object:} The unit object is $\mathbbm{1} = (K, L, \mathrm{id})$. This plays the role of the identity with respect to the tensor structure. It endows the category $\Mod^B_{K,L}$ with an internal $\ihom$ structure, and we have $H^{\vee}=\ihom(H,\mathbbm{1})$.
\end{itemize}

This category serves as the natural target for realization functors from categories of motives or similar structures, encoding comparison isomorphisms between two types of realizations via the third component $\varpi$.

\begin{defn}[Period pairing]\label{period pairing for T}
Let $\mathcal{C}$ be an additive category and $T: \mathcal{C} \to \Mod^{B}_{K,L}$ an additive functor, written as $X \mapsto (T_K(X), T_L(X), \varpi_X)$. A \emph{period pairing} for $T$ is a pair $\mathcal{H} = (F, G)$ consisting of:
\begin{enumerate}
    \item an additive covariant left exact functor $F: \mathcal{C} \to \Vect(\Q)$, which can be viewed as selecting a distinguished rational structure inside the $K$-realization;
    \item an additive contravariant left exact functor $G: \mathcal{C} \to \Vect(U)$, for some algebraic extension $U/\Q$, which analogously selects a distinguished $U$-structure in the $L$-realization;
    \item natural embeddings $F(X) \hookrightarrow T_K(X) \otimes_K B$ and $G(X) \hookrightarrow T_L(X)^\vee \otimes_L B$ for all $X \in \mathcal{C}$, reflecting their compatibility with the comparison isomorphisms $\varpi_X$.
\end{enumerate}
These embeddings are required to be functorial in $X$, ensuring the coherence of the pairing across morphisms. In other words, the diagrams
    \begin{equation*}
    \begin{tikzcd}
    F(X) \arrow[r, hook] \arrow[d, "F(f)"] & T_K(X)\otimes_K B \arrow[d, "T_{K}(f)\otimes B"] & &  G(X) \arrow[r, hook] \arrow[d, "G(f)"] & T_L\ve(X)\otimes_L B \arrow[d, "T\ve_{L}(f)\otimes B"] \\
    F(X') \arrow[r, hook]                    & T_{K}(X')\otimes_K B & &   G(X') \arrow[r, hook]                    & T_{L}\ve(X')\otimes_L B          
\end{tikzcd}
    \end{equation*}
    commute for any morphism $X\xrightarrow{f}X'$ in $\cC$.
\end{defn}

The purpose of such pairings is to isolate well-behaved subspaces of the realizations on which we can meaningfully define a bilinear pairing valued in $B$. These pairings serve as the abstract setting in which periods---viewed as numbers obtained by integrating one realization against another---can be defined and studied.

\begin{defn}[Periods and period space]\label{Space of H-periods}
Let $\mathcal{H} = (F, G)$ be a period pairing for $T: \mathcal{C} \to \Mod^B_{K,L}$.
\begin{enumerate}
    \item The $\mathcal{H}$-periods of an object $X \in \mathcal{C}$ are defined by
    \[ \mathcal{P}_\mathcal{H}(X) := \{ \gamma(\varpi_X(\nu)) \in B \mid \nu \in F(X), \gamma \in G(X) \}. \]
    This set consists of all scalar pairings between the embedded realizations of $F(X)$ and $G(X)$ under the comparison isomorphism. We set $\langle\nu,\gamma\rangle_\cH:=\gamma(\varpi_X(\nu))$.
    \item The $U$-vector space generated by $\mathcal{P}_\mathcal{H}(X)$ is denoted $\mathcal{P}_\mathcal{H}\langle X \rangle$. This space captures the linear span of periods associated to $X$.

    \item For a full additive subcategory $\mathcal{D} \subset \mathcal{C}$, we define the full set of periods arising from all objects in $\mathcal{D}$
    \[ \mathcal{P}_\mathcal{H}(\mathcal{D}) := \bigcup_{X \in \mathcal{D}} \mathcal{P}_\mathcal{H}(X). \]
\end{enumerate}
\end{defn}

This notion captures, in an abstract setting, the idea of numerical periods arising from pairings between two realizations of a given object. These periods are central objects of study in transcendence theory, arithmetic geometry, and $p$-adic Hodge theory.

\begin{remark}
\begin{enumerate}
    \item Suppose that $F(X)$ and $G(X)$ are finite-dimensional. The $\mathcal{H}$-period space $\mathcal{P}_\mathcal{H}\langle X \rangle$ is the $U$-vector space generated by the entries of all matrices $M_{S,S'}(X)$, where $S$ is a $\Q$-basis for $F(X)$ and $S'$ is a $U$-basis for $G(X)^\vee$. Each matrix $M_{S,S'}(X)$ encodes the coordinates of the elements of $S$ with respect to the dual basis $S'$ via the comparison isomorphism $\varpi_X$.

    \item The set $\mathcal{P}_\mathcal{H}(\mathcal{C})$ depends only on the images of the functors $F$ and $G$, not on their specific presentations.
    
    \item In the construction, as well as in all results, propositions, and lemmas of this section concerning the period pairing \( \cH = (F, G) \) and \( \cH \)-periods, the left exactness of \( F \) and \( G \) suffices. Nevertheless, as noted in \cref{example 4.1}, the classical periods of 1-motives fall within this formalism. In that case, \( F \) and \( G \) correspond to the singular and de Rham co-realization functors, respectively---both of which are exact.

    \item There are two fundamental types of relations among the $\mathcal{H}$-periods of an object $X \in \mathcal{C}$:
    \begin{itemize}
        \item \textbf{Bilinearity:} For all $a_1, a_2 \in \Q$, $b_1, b_2 \in U$, $\nu_1, \nu_2 \in F(X)$, and $\gamma_1, \gamma_2 \in G(X)$, we have
        \[
        \langle a_1\nu_1 + a_2\nu_2, b_1\gamma_1 + b_2\gamma_2 \rangle = a_1b_1\langle \nu_1, \gamma_1 \rangle + a_1b_2\langle \nu_1, \gamma_2 \rangle + a_2b_1\langle \nu_2, \gamma_1 \rangle + a_2b_2\langle \nu_2, \gamma_2 \rangle.
        \]

        \item \textbf{Functoriality:} For any morphism $f: X \to Y$ in $\mathcal{C}$, $\nu \in F(X)$, and $\gamma \in G(Y)$, we have
        \[
        \langle f_*\nu, \gamma \rangle = \langle \nu, f^*\gamma \rangle.
        \]
    \end{itemize}
    These natural relations motivate the construction of a formal period space that encapsulates precisely the relations expected from bilinearity and functoriality.
\end{enumerate}
\end{remark}

\begin{defn}[Space of formal $\mathcal{H}$-periods]\label{space of formal H-periods}
Let $\mathcal{H}=(F,G)$ be a period pairing for a functor $T: \mathcal{C} \to \Mod_{K,L}^B$. The space of formal $\mathcal{H}$-periods $\widetilde{\mathcal{P}}_\mathcal{H}(\mathcal{C})$ is defined as the quotient
\[
\widetilde{\mathcal{P}}_\mathcal{H}(\mathcal{C}) := \left(\bigoplus_{X \in \mathcal{C}} F(X) \otimes_{\Q} G(X)\right) \Big/ \text{(functoriality relations)}.
\]
The $\mathcal{H}$-period map is the evaluation map
\[
\mathrm{eval}_\cH: \widetilde{\mathcal{P}}_\mathcal{H}(\mathcal{C}) \to \mathcal{P}_\mathcal{H}(\mathcal{C})
\]
induced by $\nu \otimes \gamma \mapsto \gamma(\varpi(\nu))$ for $(\nu, \gamma) \in F(X) \times G(X)$. We say that all relations among $\mathcal{H}$-periods are induced by bilinearity and functoriality if $\mathrm{eval}_\cH$ is injective. Note that $\mathcal{P}_\mathcal{H}(\mathcal{C})$ is a subspace of $B$, and $\mathrm{eval}_\cH$ is always surjective.
\end{defn}

\begin{defn}\label{space of formal periods at depth i}
Let $\mathcal{H}=(F,G)$ be a period pairing as above. For $X \in \mathcal{C}$ and integer $i \geq 1$, define the space of formal $\mathcal{H}$-periods of depth $i$ as the quotient
\[
\widetilde{\mathcal{P}}^i_\mathcal{H}(X) := \frac{F(X) \otimes_{\Q} G(X)}{R_i(X)}
\]
where $R_i(X)$ is the subspace generated by expressions $\sum_{j=1}^m \nu_j \otimes \gamma_j$ for all exact sequences
\[
0 \to X' \to X^m \to X'' \to 0
\]
with $\nu_j \in F(X')$, $\gamma_j \in G(X'')$, and $m \leq i$. Define
\[
\widetilde{\mathcal{P}}^{\infty}_\mathcal{H}(X) := \colim_i \widetilde{\mathcal{P}}^i_\mathcal{H}(X), \quad \widetilde{\mathcal{P}}^i_\mathcal{H}(\mathcal{C}) := \colim_{X \in \mathcal{C}} \widetilde{\mathcal{P}}^i_\mathcal{H}(X).
\]
\end{defn}
\begin{prop}\label{prop 4.1}
 Let $(\cC, T, \cH)$ be as above. Assume that $\cC$ is an abelian category.Then
 $$\cP_\cH\langle X\rangle=\cP_\cH(\langle X\rangle),$$ where $\langle X\rangle$ is the full abelian subcategory of $\cC$ generated by $X$.
\end{prop}
\begin{proof}
If $\alpha_1$ is a $\cH$-period of $X_1$ and $\alpha_2$ is a $\cH$-period of $X_2$ in $\cC$, then $\alpha_1+\alpha_2$ is a $\cH$-period of $X_1\oplus X_2$, so $\cP_\cH(\langle X\rangle)$ is closed under the addition and scalar multiplication. If $0\to X_1\to X\to X_2\to 0$ is an exact sequence in $\cC$, we can conclude that $\cP_\cH\langle X_1\rangle +\cP_\cH\langle X_2\rangle\subseteq \cP_\cH\langle X\rangle$. This implies that for any subquotient $Y$ of $X$, $\cP_\cH(Y)\subseteq\cP_\cH\langle X\rangle$. As we showed above, we have $\cP_\cH(Y^n)\subseteq\cP_\cH\langle Y\rangle$. Since all objects in $\langle X\rangle$ are subquotients of some $X^n$ for some $n$, it follows that $\cP_\cH(\langle X\rangle)=\cP_\cH\langle X\rangle$.
\end{proof}

\begin{remark}
For a given short exact sequence $$0\to X_1\xrightarrow{i} X^n\xrightarrow{p}X_2\to 0$$ in $\cC$, let $\sum^n_{i=1}\nu_i\otimes\gamma_i\in F(X)\otimes_{\Q}G(X)$ with \[(\nu_1,\dots,\nu_n)\in i_*(F(X_1)), (\gamma_1,\dots,\gamma_n)\in p^*(G(X_2)),\]
that is, $\nu_i=i_*\nu'_i$, $\gamma_i=p^*\gamma'_i$, for some $\nu'_i\in F(X_1)$ and, $\gamma'_i\in G(X_2)$. Then, we have
$$\sum^n_{i=1}\nu_i\otimes\gamma_i=\sum^n_{i=1}i_*\nu'_i\otimes\gamma_i=\sum^n_{i=1}\nu'_i\otimes i^*p^*\gamma'_i=0.$$

Any such exact sequence as above gives rise to a non-trivial vanishing $\cH$-period, which motivates the following definition.
\end{remark}

\begin{defn}
We say the $\mathcal{H}$-period conjecture holds at depth $i$ for $\mathcal{C}$ if the evaluation map
\[
\widetilde{\mathcal{P}}^i_\mathcal{H}(\mathcal{C}) \to \mathcal{P}_\mathcal{H}(\mathcal{C})
\]
is injective.
\end{defn}

\begin{lemma}\label{colim injective period conjecture}
Let $\mathcal{H}$ be a period pairing for $T: \mathcal{C} \to \Mod^B_{K,L}$, with $\mathcal{C}$ abelian. For each object $X \in \mathcal{C}$, let $\langle X \rangle$ denote the full abelian subcategory generated by $X$.
\begin{enumerate}
    \item The evaluation map $\widetilde{\mathcal{P}}_\mathcal{H}(\mathcal{C}) \to \mathcal{P}_\mathcal{H}(\mathcal{C})$ is bijective if and only if $\widetilde{\mathcal{P}}_\mathcal{H}(\langle X \rangle) \to \mathcal{P}_\mathcal{H}\langle X \rangle$ is bijective for every $X$.
    \item The $\mathcal{H}$-period conjecture holds at depth $i$ for $\mathcal{C}$ if and only if it holds for $\langle X \rangle$ for all $X$.
\end{enumerate}
\end{lemma}
\begin{proof}
We prove (1); the proof of (2) is similar. The natural map
\[
\widetilde{\mathcal{P}}_\mathcal{H}(\langle X \rangle) \to \widetilde{\mathcal{P}}_\mathcal{H}(\mathcal{C})
\]
is injective for each $X$, hence if $\widetilde{\mathcal{P}}_\mathcal{H}(\mathcal{C}) \to \mathcal{P}_\mathcal{H}(\mathcal{C})$ is injective, then so is the composition
\[
\widetilde{\mathcal{P}}_\mathcal{H}(\langle X \rangle) \to \widetilde{\mathcal{P}}_\mathcal{H}(\mathcal{C}) \to \mathcal{P}_\mathcal{H}(\mathcal{C})
\]
for each $X \in \mathcal{C}$.

Conversely, since any morphism $f: X \to Y$ in $\mathcal{C}$ lies in the full abelian subcategory $\langle X \oplus Y \rangle$, we have
\[
\mathcal{C} = \bigcup_{X \in \mathcal{C}} \langle X \rangle.
\]
As a result,
\[
\widetilde{\mathcal{P}}_\mathcal{H}(\mathcal{C}) = \colim_{X \in \mathcal{C}} \widetilde{\mathcal{P}}_\mathcal{H}(\langle X \rangle).
\]
By assumption, the evaluation map is injective on each $\langle X \rangle$, and colimits preserve injectivity in this setting, so the global evaluation map is injective.

The surjectivity is straightforward.
\end{proof}
\begin{cor}\label{cor: P1=P for period conjecture}
Let $X \in \mathcal{C}$. Then
\[
\mathcal{P}_\mathcal{H}\langle X \rangle \cong \widetilde{\mathcal{P}}^1_\mathcal{H}(X) = \widetilde{\mathcal{P}}^2_\mathcal{H}(X) = \dots = \widetilde{\mathcal{P}}^{\infty}_\mathcal{H}(X)
\]
if and only if the evaluation map $\widetilde{\mathcal{P}}_\mathcal{H}(X) \to \mathcal{P}_\mathcal{H}(\langle X \rangle)$ is injective. In particular, all relations among $\mathcal{H}$-periods of $X$ arise from bilinearity and functoriality if and only if the period conjecture holds at depth 1.
\end{cor}
\begin{proof}
Following the same argument as \cite[Theorem 1.3]{hormann2021note}, we can show that
\[
\widetilde{\mathcal{P}}^\infty_\mathcal{H}(X) = \widetilde{\mathcal{P}}_\mathcal{H}(X).
\]
We also have successive quotients
\[
\widetilde{\mathcal{P}}^1_\mathcal{H}(X) \to \widetilde{\mathcal{P}}^2_\mathcal{H}(X) \to \dots \to \widetilde{\mathcal{P}}^\infty_\mathcal{H}(X).
\]
If the map $\widetilde{\mathcal{P}}^1_\mathcal{H}(X) \to \mathcal{P}_\mathcal{H}(\langle X \rangle)$ is injective, then so is
\[
\widetilde{\mathcal{P}}_\mathcal{H}(X) = \widetilde{\mathcal{P}}^\infty_\mathcal{H}(X) \to \mathcal{P}_\mathcal{H}(\langle X \rangle).
\]
It follows that all the maps in the above sequence are isomorphisms, and hence
\[
\mathcal{P}_\mathcal{H}(\langle X \rangle) \cong \widetilde{\mathcal{P}}^1_\mathcal{H}(X) = \widetilde{\mathcal{P}}^2_\mathcal{H}(X) = \dots = \widetilde{\mathcal{P}}_\mathcal{H}(X).
\]

Conversely, suppose that $\widetilde{\mathcal{P}}_\mathcal{H}(X) \to \mathcal{P}_\mathcal{H}(\langle X \rangle)$ is injective. Then by \cite[Lemma 3.2 and Corollary 4.2]{hormann2021note}, the evaluation map at depth 1,
\[
\widetilde{\mathcal{P}}^1_\mathcal{H}(X) \to \mathcal{P}_\mathcal{H}(\langle X \rangle),
\]
must also be injective.

\end{proof}

As the above proof demonstrates, if the $\cH$-period conjecture holds at depth $i$, then it also holds at every depth $k\geq i$. These conjectures reflect the increasing depth of functorial relationships among $\cH$-periods.

\begin{example}\label{example 4.1}
Let $\mathcal{C} = \Mi(\bar{\Q})$, the isogeny category of 1-motives over $\bar{\Q}$, and let $T: \mathcal{C} \to \Mod^\C_{\Q, \bar{\Q}}$ be the Betti-de Rham realization:
\[
M \mapsto (\Tsing(M) \otimes_{\Z} \Q, \TdR(M), \omega)
\]
where $\omega$ is the Betti-de Rham comparison isomorphism $\omega:\Tsing(M)\otimes_{\Z}\C\cong\TdR(M)\otimes_{\barQ}\C$. Define $\cH=(F,G):\Mi(\barQ)\to\Vect(\Q,\barQ)$ by
\[
F(M) := \Tsing(M) \otimes_{\Z} \Q, \quad G(M) := \TdRv(M)_{\bar{\Q}}.
\]
The $\cH$-periods of $\cC$ are the classical periods of 1-motives. In \cite[Theorem 9.7]{huber2022transcendence}, Hubber and Wüstholz showed that the depth-1 period conjecture holds for $M$, and hence so does the full conjecture: $\widetilde{\mathcal{P}}_\mathcal{H}(M) \cong \mathcal{P}_\mathcal{H}(\langle M \rangle)$, i.e, the Kontsevich-Zagier period conjecture for 1-motives over $\barQ$ holds (\cite[Theorem 9.10]{huber2022transcendence}).

\end{example}

We now turn to $p$-adic periods. Let $\gMi(K)$ denote the isogeny category of 1-motives with good reduction over $K$. Define the functor
\[
T: \gMi(K) \to \Mod^{\Bt}_{\Q_p, \bar{\Q}} \quad \text{by} \quad M \mapsto (\Vp(M), \TdR(M), \varpi \otimes \Bt)
\]
Here, $\varpi\otimes\Bt:\Tp(M)\otimes_{\Z_p}\Bt\to\TdR(M)\otimes_K\Bt$ is the isomorphism induced by the integration map $\varpi:\Tp(M)\to\TdR(M)\otimes_K\Bt$ of \cref{theorem p-adic integration pairing for M}. In the next section, we construct three period pairings---$\int^{\hp}$, $\int^{\Hp}$, and $\int^{\Hpp}$---for $T$ and show that the $\int^{\hp}$-period conjecture holds at depth 1, while the $\int^{\Hp}$- and $\int^{\Hpp}$-period conjectures hold at depth 2.


\subsection{$\Q$-Structures $\Hp$ and $\Hpp$}

In this subsection, we introduce two canonical $\Q$-structures, denoted $\Hp$ and $\Hpp$, lying respectively in $\Tp(M) \otimes \C_p$ and $\Tp(M) \otimes \Bt$. These are defined by pulling back the space $\hp(M)$ along Fontaine’s comparison map $\varphi_M$ (see \cref{Fontaine's pairing for M}) and the $p$-adic integration map $\varpi_M$ (see \cref{Integration map for M}). The resulting structures admit associated pairings whose values $\alpha \in \Bt$ are precisely the Fontaine–Messing $p$-adic periods.

Throughout this subsection, we assume that $M = [L \xrightarrow{u} G]$ is a $1$-motive over a number field $\K$ with good reduction at $p$. Fix an embedding $\K \hookrightarrow K$ into a complete discretely valued field $K$ with residue field $k$ and fraction field of characteristic $0$. The base change of $M$ to $K$ will be denoted $M_K$, which also admits a good reduction and extends to a $1$-motive $M_{\cO_K}$ over $\cO_K$. For notational convenience, we suppress the indices $K$ and $\cO_K$ when no ambiguity arises.

\medskip


The $p$-adic logarithm plays a central role in the construction of $p$-adic periods. Let $G$ be a commutative algebraic group defined over a complete $p$-adic subfield $K \subset \C_p$. The logarithm map
\[
\log_{G(K)} \colon G(K)_f \to \Lie(G(K))
\]
is a $K$-analytic homomorphism satisfying the following properties (see \cite[Chapter~III,~§7.6]{bourbaki1975elements} and \cite{zarhin1996p}):
\begin{enumerate}
    \item $G(K)_f$ is the smallest open subgroup of $G(K)$ such that the quotient $G(K)/G(K)_f$ is torsion-free;
    \item the differential $d\log_{G(K)} \colon \Lie(G)\to\Lie(\Lie(G))=\Lie(G)$ is the identity map;
    \item the logarithm map $\log_{G(K)}$ is functorial in $G$ and compatible with base change.
\end{enumerate}
These properties determine both $G(K)_f$ and $\log_{G(K)}$ uniquely. Explicitly, $G(K)_f$ consists of all elements $x \in G(K)$ for which the identity is an accumulation point of the set $\{x^n \mid n > 0\}$. That is, there exists an increasing sequence $(n_i)$ of positive integers such that $x^{n_i} \to 0$ in $G(K)$.

Let $\sG$ denote the $p$-divisible group over $\cO_K$ associated with $M = [L \to G]$, and let $\sG^0$ denote its connected component. Note that the connected component of $M[p^\infty]$ coincides with that of $G[p^\infty]$. The group $\sG^0$ is a formal $p$-divisible group of dimension $d$, represented by the formal spectrum $\Spf(\cO_K[[X_1, \dots, X_d]])$ with group law given by a formal group law
\[
F(X, Y) = X + Y + \text{(higher degree terms)} \in \cO_K[[X, Y]]^d.
\]
Consequently, the formal points $\sG^0(\cO_K)$ are identified with $\mathfrak{m}^d$, where $\mathfrak{m}$ is the maximal ideal of $\cO_K$. This endows $\sG^0(\cO_K)$ with the structure of a $p$-adic analytic subgroup of $G(K)$. The multiplication-by-$p$ map is given by a power series with linear term $px$ and higher-order terms in $\mathfrak{m}^2$. Thus, for $x \in \mathfrak{m}^d$, one has
\[
[p](x) = px + \text{(terms in } \mathfrak{m}^2),
\]
which implies $[p^n](x) \to 0$ in the $\mathfrak{m}$-adic topology. Therefore, every element of $\sG^0(\cO_K)$ is topologically nilpotent, and we conclude:
\[
\sG^0(\cO_K) \subset G(K)_f.
\]
In contrast, the étale part $\sG^{\text{ét}}$ satisfies $\sG^{\text{ét}}(\cO_K) \subset G(K)$ is torsion and discrete.

Denote by $\log_{\sG(\cO_K)}$ and $\log_{\sG^0(\cO_K)}$ the restrictions of $\log_{G(K)}$ to $\sG(\cO_K)$ and $\sG^0(\cO_K)$, respectively. Given $x \in \sG(\cO_K)$, there exists an integer $n > 0$ such that $p^n x \in \sG^0(\cO_K)$, and the logarithm satisfies
\[
\log_{\sG(\cO_K)}(x) = \frac{1}{p^n} \log_{\sG^0}(p^n x).
\]
It follows that
\[
\im(\log_{\sG(\cO_K)}) \otimes \Q = \im(\log_{\sG^0(\cO_K)}) \otimes \Q \subset \Lie(G)_K.
\]

\begin{example}\label{logarith of p-divisible multiplicative}
When $G=\mu_{p^{\infty}}=\G_m[p^{\infty}]$ we have 
$$\mu_{p^{\infty}}(\cO_L)=\Hom_{\cO_K-cont}(\cO_L[[t]],\cO_L)\cong\fm_L\cong 1+\fm_L ,$$ where the last isomorphism is given by $x\mapsto x+1$. Let $\sI:=(t)$ be the augmentation ideal of the formal p-divisible group $\mup$, we get
$$t_{\mu_{p^{\infty}}}(L)=\Hom_{\cO_K}(\sI/\sI^2,L)\cong L.$$
Thus, we have the commutative diagram 
\begin{equation*}
\begin{tikzcd}
\mup(\cO_L) \arrow[r, "\log_{\mup}"] \arrow[d, "\cong"] & t_{\mup}(L) \arrow[d, "\cong"] \\
1+\fm_L \arrow[r]                                       & L                             
\end{tikzcd}
\end{equation*}
where the left vertical arrow and right vertical arrow are given by $f\mapsto 1+f(t)$ and $g\mapsto g(t)$ respectively. Let $x\in\fm_L$ and $f\in \mup(\cO_L)$. By considering the fact that the group law on formal p-divisible group $\Spf(\cO_K[[t]])$ is induced by $\cO_K[[t]]\to\cO_K[[t,t']],\, t\mapsto (1+t)(1+t')-1$, we can write
\[
(p^nf)(t)=f([p^n]_{\mup}(t))=f((1+t)^{p^n}-1)=(1+f(t))^{p^n}-1
\]
and 
one can show that
\[
\log_{\mup}(1+x)=\sum^{\infty}_{i=1}\frac{(-1)^{i-1}x^i}{i},
\]
which coincides with the usual p-adic logarithm $\log_p:1+\fm\to L$.
\end{example}

\begin{defn}\label{definition hp}
We define
\[
\hp(M, \K) := \im\big(\log|_{\sG(\K)} \colon \sG(\cO_K) \to \Lie(G)_K\big) \otimes \Q,
\]
where $\sG(\K) := \sG(\cO_K) \cap G(\K)$. We then define the global version as the filtered colimit
\[
\hp(M) := \hp(M, \bar{\K}) := \colim_{\K' \subset \bar{\K}} \hp(M, \K'),
\]
taken over all finite extensions $\K' / \K$.
\end{defn}

\begin{defn}\label{definition Hp}
We define $\Hp(M)$ as the fibre product of $\hp(M, \bar{\K})$ and $\Tp(M) \otimes_{\Z_p} \C_p$ over $\Lie(G) \otimes_{\cO_K} \C_p(1)$, via the Fontaine's map
\[
\phi_M \otimes 1_{\C_p} \colon \Tp(M) \otimes_{\Z_p} \C_p \longrightarrow \Lie(G) \otimes_{\cO_K} \C_p(1),
\]
as in \cref{Fontaine's map for M}. This gives rise to the following cartesian diagram:
\[
\begin{tikzcd}
\Hp(M) \arrow[d] \arrow[r] & \hp(M, \bar{\K}) \arrow[d, "\iota"] \\
\Tp(M) \otimes_{\Z_p} \C_p \arrow[r, "\phi_M \otimes 1_{\C_p}"] & \Lie(G) \otimes_{\cO_K} \C_p(1)
\end{tikzcd}
\]
Here, $\iota$ is the $\Q$-linear map induced by the natural embedding
\[
\hp(M) \hookrightarrow \Lie(G) \otimes_{\cO_K} \C_p \to \Lie(G) \otimes_{\cO_K} \C_p(1), \quad x \mapsto x \otimes 1.
\]
\end{defn}

By construction, $\Hp(M)$ is a $\Q$-linear subspace of $\Tp(M)\otimes_{\Z_p}\C_p$. We have the following commutative diagram with exact rows:
\[
\begin{tikzcd}[column sep=small]
0 \arrow[r] & S \arrow[r] \arrow[d, Rightarrow, no head] & \Hp(M) \arrow[r] \arrow[d, hook] & \hp(M) \arrow[r] \arrow[d, hook] & 0 \\
0 \arrow[r] & S \arrow[r] & \Tp(M)\otimes_{\Z_p}\C_p \arrow[r] & \Lie(G)\otimes_{\cO_K}\C_p(1) \arrow[r] & 0
\end{tikzcd}
\]

By \cref{Hodge-Tate weights of tate module of 1-motives}, the Hodge–Tate weights of $\Vp(M)$ are $0$ and $1$, appearing with multiplicities $n = \operatorname{rank}(L) + \dim(A)$ and $m = \dim(T) + \dim(A)$, respectively. Hence, $S \cong \C_p^{n}$. The bottom sequence is split, and this splitting is unique, as
\[
\Ext^1(\C_p(1)^m, \C_p^n) = H^1(K, \C_p(-1)^{mn}) = 0, \quad \Hom(\C_p(1)^m, \C_p^n) = H^0(K, \C_p(-1))^{mn} = 0.
\]
Here, $\Ext^i$ is taken in the category of continuous $\Gamma_K$-representations. Consequently, by the Hodge–Tate decomposition,
\[
S = \coLie(\sG\ve)\otimes_{\cO_K}\C_p = V(M)\otimes\C_p.
\]
The second equality follows from identification \ref{identification V(M) and coLie}.

Considering the canonical exact sequence \ref{canonical exact sequence for any 1-motive}, we obtain the following commutative diagram with split exact rows:
\begin{equation}\label{diagram 4.2}
\begin{tikzcd}
& 0 & 0 & & \\
0 \arrow[r] & V(L)\otimes\C_p \arrow[r, Rightarrow, no head] \arrow[u] & \Hp(L) \arrow[r] \arrow[u] & 0 & \\
0 \arrow[r] & V(M)\otimes\C_p \arrow[u] \arrow[r] & \Hp(M) \arrow[u] \arrow[r] & {\hp(M,\bar{\K})} \arrow[u] \arrow[r] & 0 \\
0 \arrow[r] & V(G)\otimes\C_p \arrow[r] \arrow[u] & \Hp(G) \arrow[r] \arrow[u] & {\hp(G,\bar{\K})} \arrow[u, Rightarrow, no head] \arrow[r] & 0 \\
& 0 \arrow[u] & 0 \arrow[u] & 0 \arrow[u] &
\end{tikzcd}    
\end{equation}

To complement the construction of $\Hp(M)$ via Fontaine’s comparison map, we now define a second $\Q$-structure within $\Tp(M) \otimes \Bt$ by pulling back $\hp(M)$ along the $p$-adic integration map $\varpi_M$. This yields the space $\Hpp(M)$ of periods arising from $p$-adic integration.
\begin{defn}\label{definition Hpp}
We define $\Hpp(M)$ as the pullback of $\Hp(M) \hookrightarrow \Tp(M)\otimes_{\Z_p}\C_p$ along the canonical surjection
\[
\Tp(M)\otimes_{\Z_p}\Bt \longrightarrow \Tp(M)\otimes_{\Z_p}\C_p.
\]
\end{defn}

Thus, the spaces $\Hp(M)$ and $\Hpp(M)$ define two canonical $\Q$-structures on the $p$-adic realizations of the 1-motive $M$, each encoding arithmetic information via Fontaine's comparison isomorphism and $p$-adic integration, respectively.

We have the following commutative diagram:
\begin{equation}\label{7.2.6}
\begin{tikzcd}
0 \arrow[r] & \Tp(M)\otimes_{\Z_p}\C_p(1) \arrow[r] \arrow[d, equal] & \Hpp(M) \arrow[r] \arrow[d, hook] & \Hp(M) \arrow[r] \arrow[d, hook] & 0 \\
0 \arrow[r] & \Tp(M)\otimes_{\Z_p}\C_p(1) \arrow[r] & \Tp(M)\otimes_{\Z_p}\Bt \arrow[r] & \Tp(M)\otimes_{\Z_p}\C_p \arrow[r] & 0
\end{tikzcd}
\end{equation}
with exact rows. The bottom sequence is obtained by tensoring the canonical exact sequence
\[
0 \to \C_p(1) \to \Bt \to \C_p \to 0
\]
with the $\Z_p$-module $\Tp(M)$.

\begin{remark}
Let $\langle \,,\, \rangle$ denote the canonical pairing
\[
\Lie(G\nat)_{\Bt} \times \coLie(G\nat)_{\Bt} \longrightarrow \Bt.
\]
The splitting of the middle row in diagram \ref{diagram 4.2} implies that $\Hp(M)$ embeds canonically into $\Lie(G\nat)_{\C_p}$. Therefore, for any $x \in \Hp(M)$ and $\omega \in \coLie(G)$, the pairing satisfies
\[
\langle x, \omega \rangle = \int^{\phi}_x \omega = \int^{\varpi}_x \omega,
\]
where the second equality follows from the compatibility of the $p$-adic integration pairing with the Hodge filtration (see \cref{theorem p-adic integration pairing for M}).

This compatibility is further illustrated in the commutative diagram below:
\[
\begin{tikzcd}
0 \arrow[r] & \Tp(M)\otimes_{\Z_p}\C_p(1) \arrow[r] \arrow[d, equal] & \Hpp(M) \arrow[r] \arrow[d, hook] & \Hp(M) \arrow[r] \arrow[d, hook] & 0 \\
0 \arrow[r] & \Tp(M)\otimes_{\Z_p}\C_p(1) \arrow[r] & \Tp(M)\otimes_{\Z_p}\Bt \arrow[r] \arrow[dd, "\varpi_M\otimes\Bt"] & \Tp(M)\otimes_{\Z_p}\C_p \arrow[r] \arrow[d, "\varphi_M \otimes \C_p", two heads] & 0 \\
& & & \Lie(G)\otimes_K \C_p(1) \arrow[d, hook] & \\
0 \arrow[r] & V(M)\otimes_K \Bt \arrow[r] & \Lie(G\nat)\otimes_K \Bt \arrow[r] & \Lie(G)\otimes_K \Bt \arrow[r] & 0
\end{tikzcd}
\]
The top squares are commutative by construction, while the bottom square commutes by diagram \ref{commutative diagram respect filtration}.

Suppose now that $x \in \Hpp(M)$. Then $x$ can be written as $x = u + v$, where the image of $u$ lies in $\Hp(M)$ and $v \in \Tp(M)\otimes_{\Z_p}\C_p(1)$, hence $v$ is in the kernel of the projection $\Tp(M)\otimes_{\Z_p}\Bt \to \Tp(M)\otimes_{\Z_p}\C_p$. Although such a decomposition is not unique, the diagram above ensures that $v$ maps to an element of $V(M)\otimes_K\Bt$. It follows that for all $\omega \in \coLie(G)$,
\[
\int^{\varpi}_x \omega = \int^{\varpi}_u \omega + \int^{\varpi}_v \omega = \int^{\varpi}_u \omega + 0 = \int^{\phi}_u \omega.
\]
\end{remark}

As a consequence, we have

\begin{cor}\label{comparing parings Hp and Hpp}
Let $\omega \in \coLie(G)_{\bar{\K}}$, and let $\pi : \Hpp(M) \to \Hp(M)$ denote the canonical projection. Then
\[
\int^{\varpi}_x \omega = \int^{\phi}_{\pi(x)} \omega.
\]
\end{cor}

The $\Q$-structure $\Hpp(M)$ is large, containing $\Tp(M)\otimes_{\Z_p}\C_p(1)$, whose image under $\int^{\varpi}$ is generally unknown. To address this, we  define a more refined $\Q$-structure inside $\Hpp(M)$ using the crystalline structure of $\Bt$. Let $\Bct := \Bcrisp/t^2\Bcrisp \subseteq \Bt$. Since the Frobenius $\phi_{\mathrm{cris}}$ is injective on $\Bcrisp$, and acts on $t$ by multiplication by $p$, it induces an injective endomorphism on $\Bct$. We thus define:

\begin{defn}\label{definition of tilde of Hpp}
Recall the definition of $\tilde{\Tp}(M)$ from \ref{tilde Tp and Tcrys}. We define $\tilde{\Hpp}(M)$ to be the pullback of the inclusion $\Hpp(M) \hookrightarrow \Tp(M)\otimes_{\Z_p}\Bt$ along the natural map
\[
\tilde{\Tp}(M)\otimes_{\Z_p}\Brigt  \hookrightarrow \Tp(M)\otimes_{\Z_p}\Bt,
\]
that is,
\begin{equation}
\begin{tikzcd}
\tilde{\Hpp}(M) \arrow[d, hook] \arrow[r, hook] & \Hpp(M) \arrow[d, hook] \\
\tilde{\Tp}(M)\otimes_{\Z_p}\Brigt \arrow[r, hook] & \Tp(M)\otimes_{\Z_p}\Bt
,\end{tikzcd}
\end{equation}
where $\Brigt:=\bigcap^{\infty}_{m=1}\phi^m_{\mathrm{cris}}(\Bct)$ is the rigidified subring of $\Bct$.
\end{defn}
\begin{prop}\label{Brigt cap Cp(1)}
We have $\C_p(1)\cap \Brigt=0$.
\end{prop}
\begin{proof}
Note that \( \C_p(1) \) is (non-canonically) isomorphic to the \(\C_p\)-line \( \C_p \cdot t \), and its image in \( \Bct \) is given by the class \( \bar{t} := t \mod t^2 \Bcrisp \). Hence, $\C_p=\Bct/\bar{t}\Bct$. Observe that:
\[
\varphi_{\mathrm{cris}}(\bar{t}) = \varphi_{\mathrm{cris}}(t) \mod t^2 = pt \mod t^2 = p \cdot \bar{t},
\]
and hence by iteration,
\[
\varphi^m_{\mathrm{cris}}(\bar{t}) = p^m \cdot \bar{t}.
\]

Suppose now that \( x \in \C_p(1) \cap \bigcap_{m=1}^\infty \varphi^m_{\mathrm{cris}}(\Bct) \). Then \( x\in t\Bcrisp/t^2\Bcrisp\), and for all \( m \), there exists \( x_m \in \Bct \) such that
\[
\varphi^m_{\mathrm{cris}}(x_m) = x.
\]
Since $\phi_{\mathrm{cris}}$ is injective, we can conclude that $x_m\in\bar{t}\Bct$. Hence, $x=\phi^m_{\mathrm{cris}}(x_m)\in p^m\Bct$. It follows that
\[x\in\bigcap_{m\geq 1}p^m\Bct=0, \]
since $\Bct$ is p-adically separated, being a quotient of the p-adically separated ring $\Bcrisp$.
\end{proof}

To define the relevant period pairings with respect to the $\Q$-structures $\Hp$ and $\Hpp$, we must refine the target space accordingly. The description of $\hp(M)$ suggests that its elements arise from the connected component $\sG^0$ of $\sG$, whose associated isocrystal has non-zero slopes. Consequently, it is natural—and indeed necessary, as shown in the proof of our main theorem (\cref{level 2 period cinjecture for Hpp})—to pair them with differential forms that also correspond to the non-zero slope part of the isocrystal associated to $M$. Otherwise, we may encounter some vanishing periods, making it difficult to trace the underlying relations.

Let $N$ be a non-zero isocrystal over $K_0$. By \cite{simpleproofofDieudonne-Maninclassification}, it admits a unique decomposition
\[
N = \bigoplus N(\alpha_i),
\]
where each $N(\alpha_i)$ is a nonzero isoclinic sub-isocrystal of slope $\alpha_i$. We define $\tilde{N}$ to be the direct sum of the non-zero slope components:
\[
\tilde{N} := \bigoplus_{\alpha_i \neq 0} N(\alpha_i) \subset N.
\]

Let $\K'$ be a finite extension of $\K$ in which $p$ remains unramified, and let $k'$ be the residue field at the prime above $p$. Then, by the crystalline–de Rham comparison isomorphism (\cref{crystalline-de Rham comparison isomorphism}), we obtain a canonical identification
\[
\TdR\ve(M)\otimes_{\K} K' \cong \Tcrys\ve(\bar{M})\otimes_{\W(k')} K',
\]
where $K'$ is the $p$-adic completion of $\K'$. Let $K'_0$ denote the fraction field of $\W(k')$, and write $N := \Tcrys\ve(\bar{M})\otimes_{\W(k')} K'_0$. Let $\tilde{N}_{K'_0} \subset N$ denote the subobject of non-zero slopes as above. We then define $\tilde{N}(\K')$ as the fibre product
\begin{equation}\label{eq 6.2.8}
\begin{tikzcd}
\tilde{N}(\K') \arrow[d, hook] \arrow[r, hook] & \tilde{N}_{K'_0} \arrow[d, hook] \\
\TdRv(M)_{\K'} \arrow[r, hook] & \Tcrysv(\bar{M})\otimes_{\W(k')} K'
\end{tikzcd}
\end{equation}
where the bottom map is the composition with  the canonical comparison isomorphism
\[
\TdRv(M)_{\K'} \hookrightarrow \TdRv(M)\otimes_{\K} K' \xrightarrow{\sim} \Tcrys\ve(\bar{M})\otimes_{\W(k')} K'.
\]

The action of Frobenius is compatible with unramified extension of local fields, and this construction is compatible with unramified extension of $\K$ at $p$. We define
\begin{equation}
\tilde{N}(M) := \varinjlim_{\K'} \tilde{N}(\K'), 
\end{equation}
where the colimit is taken over all finite extensions $\K'$ of $\K$ such that $p$ remains unramified. Equivalently, the colimit may be taken over all finite extensions $\K'$ satisfying $\K \subseteq \K' \subseteq \K^u := K^{\mathrm{ur}} \cap \bar{\K}$, where $\K^u$ denotes the maximal unramified extension of $\K$ at $p$. This extension is infinite and Galois over $\K$.

Since every element $\omega \in \Lie(G\nat)_{\bar{\K}}$ belongs to $\Lie(G\nat)_{\K'}$ for some finite $\K'$, the space $\tilde{N}(M)$ embeds into
\[
\TdRv(M_{\K}) \otimes \bar{K} \hookrightarrow \TdRv(M_K)\otimes_K \Bt = \coLie(G\nat)\otimes_K \Bt.
\]
Thus, $\tilde{N}(M)$ is naturally a $\K^u$-linear subspace of $\coLie(G\nat)_{\bar{\K}}$.

\begin{defn}\label{three pairings hp Hp Hpp}
We define the following restricted period pairings:
\begin{itemize}
    \item The pairing induced by $\phi_M$ on $\Hp(M)$:
    \[
    \int^{\Hp} : \Hp(M) \times \coLie(G)_{\bar{\K}} \longrightarrow \C_p(1).
    \]
    \item The pairing induced by $\varpi_M$ on the refined space:
    \[
    \int^{\Hpp} : \tilde{\Hpp}(M) \times \tilde{N}(M) \longrightarrow \Bt.
    \]
    \item The restriction of $\int^{\Hp}$ to $\hp(M,\bar{\K})$:
    \[
    \int^{\hp} : \hp(M,\bar{\K}) \times \coLie(G)_{\bar{\K}} \longrightarrow \C_p(1).
    \]
\end{itemize}
We say that a $p$-adic number $\alpha \in \Bt$ is an $\Hpp$-period, $\Hp$-period, or $\hp$-period of a $1$-motive $M \in \gMi(\K)$ if it lies in the image of the corresponding pairing.
\end{defn}

\begin{remark}
Let $T : \gMi \to \Mod_{\Q,\bar{\Q}}^{B_2}$ be defined by
\[
M \mapsto \left( \Vp(M),\, \TdR(M)_{\bar{\Q}},\, \varpi \otimes \Bt \right).
\]
The pairing $\int^{\Hpp}$ defines a period pairing for $T$ in the sense of \cref{period pairing for T}, with
\[
F(M) := \tilde{\Hpp}(M), \qquad G(M) := \tilde{N}(M).
\]
We refer to any element of the image of $\int^{\Hpp}$ as an \emph{$\Hpp$-period}. The $\K^u$-vector space they span is denoted by $\cP_{\Hpp}\langle M \rangle$ (see \cref{Space of H-periods}). The corresponding formal period spaces are denoted by $\tilde{\cP}_{\Hpp}(M)$ and $\tilde{\cP}^i_{\Hpp}(M)$ for depth $i$ (\cref{space of formal H-periods}, \cref{space of formal periods at depth i}).

Similarly, the pairing $\int^{\Hp}$ defines a period pairing with
\[
F'(M) := \Hp(M), \qquad G'(M) := \coLie(G)_{\bar{\Q}},
\]
yielding the space of $\Hp$-periods $\cP_{\Hp}\langle M \rangle \subset \C_p(1)$. We denote the formal analogues by $\tilde{\cP}_{\Hp}(M)$ and $\tilde{\cP}^i_{\Hp}(M)$. Analogous notations apply to $\hp$-periods.
\end{remark}

\begin{prop}
The assignment $\tilde{N} : \gMi(\bar{\K}) \to \Vect(\bar{\K})$, given by $M \mapsto \tilde{N}(M)$, defines a contravariant exact functor.
\end{prop}
\begin{proof}
Taking the functor $\tilde{(.)}$ from an exact sequence of isocrystals over $K_0$ results in an exact sequence of isocrystals with non-zero slopes.
Since $\TdR$ and base change are exact, and pullbacks preserve exactness, the result follows.
\end{proof}


\section{P-adic subgroup theorem for 1-motives and p-adic period conjectures}

This section develops a $p$-adic analogue of the analytic subgroup theorem in the context of 1-motives with good reduction. Our main result, the $p$-adic subgroup theorem for 1-motives, asserts that for any period class $x$ in the $\Hp$-realization of a 1-motive $M$, there exists a short exact sequence of 1-motives refining the support of $x$ and capturing the annihilator of $x$ within the de Rham realization. This result provides the structural input necessary to prove the depth-2 period conjecture for $\Hpp$, the depth-2 period conjecture for $\Hp$, and the depth-1 period conjecture for $\hp$.

We begin by formulating the $p$-adic subgroup theorem and discussing its relation to the $p$-adic analytic subgroup theorem of Bertrand and Fuchs.

\subsection{Statement and Reformulation of the Subgroup Theorem}

We first set up the linear algebra formalism needed to state the subgroup theorem.

\begin{defn}\label{def: left-right kernel duality pairing}
Let $\fg$ be a finite-dimensional vector space over a field $F$, equipped with a perfect pairing $\langle \cdot,\cdot \rangle \colon \fg \times \fg^\vee \to F$. For any subspace $\fa \subseteq \fg$ and $\fb \subseteq \fg^\vee$, we define the annihilators
\[
\Ann(\fa) := \{ f \in \fg^\vee \mid \langle a, f \rangle = 0 \text{ for all } a \in \fa \}, \quad \Ann(\fb) := \{ u \in \fg \mid \langle u, b \rangle = 0 \text{ for all } b \in \fb \}.
\]
We write $\Ann(u) := \Ann(\{u\})$ for $u \in \fg$.
\end{defn}

By definition and maximality of $\Ann(u)$, we have
\[\Ann\Ann\Ann(u) = \Ann(u).
\]

In what follows, when $x \in \Hp(M)$, we view it as an element of $\Lie(G^\natural)_{\C_p}$ via diagram~\ref{7.2.6}, and interpret $\Ann(x)$ accordingly in $\coLie(G^\natural)_{\C_p}$.

We now state our main structural result.

\begin{thm}[P-adic subgroup theorem for 1-motives]\label{p-adic subgroup theorem for Fontaine pairing}
Let $M$ be a 1-motive over a number field $\K$ with good reduction at $p$, and let $x \in \Hp(M)$. Then there exists an exact sequence
\[
0 \to M_1 \to M^n \to M_2 \to 0,
\]
of 1-motives over a finite extension of $\K$, all with good reduction at $p$, and $n \in \{1,2\}$, such that $x \in \Hp(M_1)$ and $\Ann(x) \subseteq \TdRv(M_2)$.
\end{thm}

The proof relies on the $p$-adic analytic subgroup theorem, originally due to Bertrand \cite{bertrand1985lemmes} and proved in full generality by Fuchs \cite{fuchs2015p}, which serves as the $p$-adic analogue of the classical theorem of W\"ustholz \cite{wustholz1989algebraische}:

\begin{thm}[P-adic analytic subgroup theorem \cite{bertrand1985lemmes, fuchs2015p}]\label{original p-adic analytic subgroup theorem}
Let $G$ be a commutative algebraic group over $\bar{\Q}$, and let $V \subseteq \Lie(G)$ be a non-zero $\bar{\Q}$-linear subspace. If $\gamma \in G(\bar{\Q})_f$ satisfies $\log_G(\gamma) \in V \otimes_{\bar{\Q}} \C_p$ and $\log_G(\gamma) \ne 0$, then there exists an algebraic subgroup $H \subseteq G$ defined over $\bar{\Q}$ such that $\Lie(H) \subseteq V$ and $\gamma \in H(\bar{\Q})$.
\end{thm}

We begin with the following auxiliary lemma.

\begin{lemma}\label{lemma annihilator for exact sequence}
Let $0 \to \fh \to \fg \xrightarrow{\pi} \fg/\fh \to 0$ be an exact sequence of Lie algebras. Then
\[
\Ann(\fh) = \pi^*((\fg/\fh)^\vee).
\]
\end{lemma}
\begin{proof}
An element $f \in \fg^\vee$ lies in $\Ann(\fh)$ if and only if its restriction to $\fh$ vanishes, i.e., $f \in \ker(\fg^\vee \to \fh^\vee)$, which is precisely $\pi^*((\fg/\fh)^\vee)$.
\end{proof}

This allows us to formulate a more usable version of the p-adic subgroup theorem:

\begin{prop}\label{general reformulation of p-adic analytic subgroup theroem for algebraic groups}
Let $G$ be a connected commutative algebraic group over a number field $\K$, and let $u_1, \dots, u_n \in \log_G(G(\bar{\Q})_f)$ (or $u_i \in \log_{\sG}(\sG(\bar{\Q}))$ if $G$ is semi-abelian). Then there exists an exact sequence
\[
0 \to H_1 \to G \to H_2 \to 0
\]
of connected commutative algebraic groups over a finite extension of $\K$, such that:
\begin{itemize}
    \item $u_i \in \Lie(H_1)_{\C_p}$ for all $i$;
    \item $\Ann(u_1, \dots, u_n) = \coLie(H_2)_{\C_p}$.
\end{itemize}
Moreover, this exact sequence is uniquely determined by these properties.
\end{prop}

\begin{proof}
We begin with the case $n = 1$. Let $u = \log(\gamma)$ for some $\gamma \in G(\bar{\Q})_f$. If $u = 0$, the claim holds with $H_1 = 0$. Otherwise, consider the subspace $V := \Ann\Ann(u) \subset \Lie(G)_{\C_p}$. By the $p$-adic analytic subgroup theorem (\cref{original p-adic analytic subgroup theorem}), there exists a connected algebraic subgroup $H_1 \subset G$ defined over a finite extension $\K'/\K$ such that $u \in \Lie(H_1)_{\C_p} \subset V$ and $\gamma \in H_1(\bar{\Q})$.

Taking annihilators, we obtain
\[
\Ann(u) \subseteq \Ann(\Lie(H_1)) \subseteq \Ann\Ann\Ann(u) = \Ann(u),
\]
so equalities hold throughout, and we conclude $\Ann(u) = \Ann(\Lie(H_1))$. Letting $H_2 := G/H_1$, we obtain the exact sequence
\[
0 \to H_1 \to G \to H_2 \to 0
\]
over a finite extension of $\K$. From the exact sequence of Lie algebras
\[
0 \to \Lie(H_1) \to \Lie(G) \xrightarrow{\pi} \Lie(H_2) \to 0,
\]
\cref{lemma annihilator for exact sequence} yields
\[
\Ann(\Lie(H_1)) = \pi^*\left(\Lie(H_2)^\vee\right) = \coLie(H_2),
\]
which we view as a subspace of $\coLie(G)$ as $\pi^*$ is injective.

Now consider the general case $n > 1$. Apply the above argument to each $u_i$ to obtain subgroups $H_1^{(i)} \subset G$ with $u_i \in \Lie(H_1^{(i)})_{\C_p}$ and $\Ann(u_i) = \coLie(G/H_1^{(i)})_{\C_p}$. Let $H_1 := H_1^{(1)} + \dots + H_1^{(n)}$ be the algebraic subgroup they generate. Then
\[
u_i \in \Lie(H_1)_{\C_p} \quad \text{for all } i,
\]
and
\[
\coLie(G/H_1)_{\C_p} = \bigcap_i \coLie(G/H_1^{(i)})_{\C_p} = \bigcap_i \Ann(u_i) = \Ann(u_1, \dots, u_n),
\]
as the intersection is taken in $\coLie(G)$. The resulting sequence
\[
0 \to H_1 \to G \to H_2 := G/H_1 \to 0
\]
satisfies the required conditions.

To prove uniqueness, suppose another exact sequence
\[
0 \to H_1' \to G \to H_2' \to 0
\]
also satisfies the same properties. Then
\[
\coLie(H_2)_{\C_p} = \Ann(u_1, \dots, u_n) = \coLie(H_2')_{\C_p},
\]
so $\Lie(H_1) = \Lie(H_1')$, and since both $H_1$ and $H_1'$ are connected, it follows that $H_1 = H_1'$.
\end{proof}

\subsection{Proof of the Subgroup Theorem}

We now provide the full proof of the $p$-adic subgroup theorem for 1-motives (\cref{p-adic subgroup theorem for Fontaine pairing}), following the reformulated version of the analytic subgroup theorem in Proposition~\ref{general reformulation of p-adic analytic subgroup theroem for algebraic groups}.

Without loss of generality, we may assume that the morphism $L \to G$ is injective. Indeed, one can write
\[
M = [L' \to 0] \oplus [L/L' \pmod{\text{torsion}} \to G]
\]
in the isogeny category of $\gMi(\K)$, where the second summand has an injective morphism. For $M = [L' \to 0]$, we have $\TdR(L')_{\C_p} = V(L')_{\C_p} = \Hp(L')$, and one can take the exact sequence $0\to L'\to L'^2\to L'\to 0$. Hence, we assume $L \to G$ is injective.

Let $x \in \Hp(M)$. We can write $x = u + v$ with $u \in \hp(M, \bar{\K})$ and $v \in V(M)$. By the construction of $\hp(M, \bar{\K})$ (\cref{definition hp}) there exists $\gamma \in \sG(\bar{\K})$ such that $\log_{\sG}(\gamma) = u$ (up to a rational scalar). Since $G^\natural$ is a vector extension of $G$, we may view $u$ in the image of $\log_{G^\natural}$. Note that the logarithm on $V(M)$ is the identity, and we apply Proposition~\ref{general reformulation of p-adic analytic subgroup theroem for algebraic groups} to $u + v \in \Lie(G^\natural)_{\C_p}$ to obtain an exact sequence
\[
0 \to H_1 \to G^\natural \to H_2 \to 0
\]
over a finite extension of $\K$ such that $u + v \in \Lie(H_1)_{\C_p}$ and $\Ann(u + v) = \coLie(H_2)_{\C_p}$.

By the structure theory of commutative algebraic groups (\cite{conrad_modern_2002}), we have exact sequences
\[
0 \to V_1 \to H_1 \to G_1 \to 0, \quad 0 \to V_2 \to H_2 \to G_2 \to 0,
\]
where $G_i$ is semi-abelian and $V_i$ a vector group. Note that $v \in V_1\otimes\C_p$. This yields the following commutative diagram with exact rows and columns:

\begin{equation}
\begin{tikzcd}
        & 0 \arrow[d]             & 0 \arrow[d]               & 0 \arrow[d]             &   \\
0 \arrow[r] & V_1 \arrow[r] \arrow[d] & V(M) \arrow[r] \arrow[d]     & V_2 \arrow[r] \arrow[d] & 0 \\
0 \arrow[r] & H_1 \arrow[r] \arrow[d] & G^\natural \arrow[r] \arrow[d] & H_2 \arrow[r] \arrow[d] & 0 \\
0 \arrow[r] & G_1 \arrow[r] \arrow[d] & G \arrow[r] \arrow[d]     & G_2 \arrow[r] \arrow[d] & 0 \\
        & 0                       & 0                         & 0                       &  
\end{tikzcd}
\end{equation}

Define $L_1 := L \cap H_1$, and let $L_2 := L / L_1 \pmod{\text{torsion}}$. By construction,
$L_1\to L\to G$ factors through $L_1\to G_1$ and $L\to G\to G_2$ factors through $L_2\to G_2$. This yields an exact sequence
\begin{equation}\label{exact seq 5.2}
0 \to M_1 \to M \to M_2 \to 0,
\end{equation}
with $M_1 = [L_1 \to G_1]$ and $M_2 = [L_2 \to G_2]$. Since $M$ has good reduction at $p$, so do $M_1$ and $M_2$. The exact sequence \ref{exact seq 5.2} lifts to $\cO_K$, where $K$ is a finite extension of $\Q_p$ containing $\K$. By the property of universal vector extensions, the exact sequence \[0\to V_i\to [L_i\to H_i]\to [L_i\to G_i]\to 0 \]
is the push-out of $G\nat_i$ for $i=1,2$. The compositions $G\nat_1\to H_1\to G\nat$ and $G\nat\to G\nat_2\to H_2$ are injective and surjective respectively. So, the maps $G\nat_1\to H_1$ and $G\nat_2\to H_2$ are injective and surjective respectively. Since $H_1$ is a vector extension of $G_1$, we deduce that $u\in\hp(G_1,\bar{\K})$. 

We now distinguish two cases. If $v \in V(M_1)_{\C_p}$, then $x = u + v \in \Hp(M_1)$. The exact sequence \eqref{exact seq 5.2} gives $\TdRv(M_2)_{\C_p} = \coLie(G_2^\natural)_{\C_p} \subseteq \Ann(x)$. However, $G\nat_2\twoheadrightarrow H_2$ is surjective, hence 
\begin{equation}\label{ann(u)=TdR(M2)}
\Ann(u+v)=\coLie(H_2)_{\C_p}\subseteq\coLie(G\nat_2)_{\C_p}\subseteq\TdRv(M_2)_{\C_p},      
\end{equation}
as desired.

Now, assume that $v\notin V(M_1)_{\C_p}$. From \ref{identification V(M) and coLie}, we know that $V(M_1)=\coLie(G\ve_1)$ and $V(M)=\coLie(G\ve)$. Choose a semi-abelian variety $N$ such that
\begin{enumerate}
    \item There exists a surjection $N\onto G\ve_1$ that the quotient map $G\ve\onto G_1\ve$ factors through it. In particular, its Cartier dual $N\ve$ is a subgroup of $G$.
    \item $v\in\coLie(N)_{\C_p}$.
\end{enumerate}
To construct such $N$, define $$N:=G\ve/(K_1\cap K_2),$$ where $K_1$ is a subgroup of $G\ve$ whose Lie algebra is contained in $\Ann(v)$, possibly after extending scalars to \( \C_p \), and $K_2$ is the kernel of $G\ve\onto G\ve_1$. Since $K_1\cap K_2$ is a closed subgroup and $G\ve$ is a semi-abelian scheme, the quotient $N=G\ve / (K_1 \cap K_2)$ must be a semi-abelian scheme. Dualizing, we obtain $G_1 \hookrightarrow N^\vee \hookrightarrow G$, so $u \in \hp(N^\vee)$. 
Therefore, we obtain the diagram 
\begin{equation}\label{7.3.4}
\begin{tikzcd}
0 \arrow[r] & H_1 \arrow[r] \arrow[d]       & G\nat \arrow[r] \arrow[d]                  & H_2 \arrow[r] \arrow[d, two heads] & 0 \\
0 \arrow[r] & G_1 \arrow[d, hook] \arrow[r] & G \arrow[d, Rightarrow, no head] \arrow[r] & G_2 \arrow[d, two heads] \arrow[r] & 0 \\
0 \arrow[r] & N\ve \arrow[r]                & G \arrow[r]                                & G/N\ve \arrow[r]                   & 0
.\end{tikzcd}
\end{equation}
Define $M_1' = [L_1 \to N^\vee]$ and $M_2' = [L_2 \to G/N^\vee]$, yielding the diagram:

\begin{equation}\label{diagram 435}
\begin{tikzcd}
0 \arrow[r] & M_1 \arrow[r] \arrow[d, hook] & M \arrow[r] \arrow[d, Rightarrow, no head] & M_2 \arrow[r] \arrow[d, two heads] & 0 \\
0 \arrow[r] & M_1' \arrow[r]                & M \arrow[r]                                & M_2' \arrow[r]                     & 0
\end{tikzcd}
\end{equation}

Based on the condition imposed on $N$, we have $v\in\coLie(N)=V(M'_1)_{\C_p}$ and $u\in\hp(M'_1)$. Consequently, $x=u+v\in\Hp(M'_1)$, and from the above diagram \ref{diagram 435}, we have
$$
\TdRv(M'_2)_{\C_p}\subseteq\Ann(x)\subseteq\TdRv(M_2)_{\C_p},
$$ where the latter inclusion is as in \ref{ann(u)=TdR(M2)}.
The direct sum $M' := M_1 \oplus M_1'$ and $M'' := M_2 \oplus M_2'$ give an exact sequence
\[
0 \to M' \to M^2 \to M'' \to 0,
\]
with $x \in \Hp(M')$ and $\Ann(x) \subseteq \TdRv(M'')_{\C_p}$. This completes the proof.

\qed

In the course of proving the \( p \)-adic subgroup theorem for 1-motives, we get the following:
\begin{cor}\label{cor 5.6.1}
Let $x \in \Hp(M)$. We can write $x = u + v$ uniquely with $u \in \hp(M, \bar{\K})$ and $v \in V(M) \otimes \C_p$. Then there exists a commutative diagram
\[
\begin{tikzcd}
0 \arrow[r] & M_1 \arrow[r] \arrow[d, hook] & M \arrow[r] \arrow[d, Rightarrow, no head] & M_2 \arrow[r] \arrow[d, two heads] & 0 \\
0 \arrow[r] & M_1' \arrow[r]                & M \arrow[r]                                & M_2' \arrow[r]                     & 0
\end{tikzcd}
\]
with the following properties:
\begin{enumerate}
  \item $u \in \hp(M_1)$,
  \item $x \in \Hp(M_1')$,
  \item $\Ann(x) = \TdRv(M_2)_{\C_p} \supseteq \TdRv(M_2')_{\C_p}$,
  \item If $\Hp(M) = \hp(M)$ (or $x \in \hp(M)$), then \cref{p-adic subgroup theorem for Fontaine pairing} holds with $n = 1$.
\end{enumerate}
\end{cor}

\subsection{Applications to Period Conjectures}

We now apply the $p$-adic subgroup theorem to prove several depth-specific cases of the period conjectures associated with three main pairings we obtained in \cref{three pairings hp Hp Hpp}. These include the $\Hpp$-period conjecture at depth $2$, the $\Hp$-period conjecture at depth $2$, and the $\hp$-period conjecture at depth $1$.

Let $N$ be a nonzero isocrystal over $K$ with slopes $\{\alpha_1 < \cdots < \alpha_n\}$ occurring with multiplicities $\{\mu_1, \dots, \mu_n\}$. Recall that the Newton number $t_N(N)$ is defined as the $y$-coordinate of the rightmost endpoint of the Newton polygon associated with $N$, and is given by
\[
t_N(N) = \sum_i \mu_i \alpha_i.
\]
On the other hand, the Hodge number $t_H(N)$ is defined by
\[
t_H(N) := \sum_{i \in \Z} i \cdot \dim_K \gr^i(N),
\]
where the grading is taken with respect to the Hodge filtration on $N$.

A filtered isocrystal $N$ over $K$ is said to be weakly admissible if the inequality
\[
t_N(N') \geq t_H(N')
\]
holds for every sub-filtered isocrystal $N' \subset N$, with equality when $N' = N$. Furthermore, any weakly admissible filtered isocrystal is admissible if it is an essential image of $\Dcris$; see \cite{fontaine1994corps} and \cite{fargues_courbes_2018} for more details.

We now turn to the proof of the $\Hpp$-period conjecture at depth $2$. The following lemma provides a key linear algebraic input concerning the generation of an admissible filtered isocrystal by its degree one Hodge component under Frobenius:

\begin{lemma}\label{lemma 5.6.2}
Let $N$ be an admissible filtered isocrystal over $K_0$ with nonzero slopes, and suppose its filtration satisfies
\[
\Fil^i(N) = \begin{cases}
  N & \text{for } i \leq 0, \\
  X & \text{for } i = 1, \\
  0 & \text{for } i \geq 2.
\end{cases}
\]
Then we have:
\[
\sum_{n \geq 0} F^n(X) = N.
\]
\end{lemma}

\begin{proof}
We first claim that $Y := \sum_{n \geq 0} F^n(X)$ is weakly admissible. Since $N$ is admissible, any subobject $Y$ satisfies $t_N(Y) \geq t_H(Y)$. Let $\dim X = r$ and $\dim N = s$. The Hodge polygon of $N$ has a horizontal segment of length $s - r$ and a slope $1$ segment of multiplicity $r$. As the filtration of $Y$ is given by
    \begin{equation*}
          \Fil^i(Y)=\begin{cases}
            Y,\,\, i\leq 0\\
            X,\,\, i=1\\
            0,\,\, i\geq 2,
        \end{cases}
    \end{equation*}
the Hodge polygon of $Y$ has horizontal segment of length $\dim Y - r$ and a segment of slope $1$ with multiplicity $r$. Since $N$ and $Y$ have only positive slopes, we get $t_N(Y) \leq t_N(N) = t_H(N) = r$. Therefore, $Y$ is weakly admissible, and so $N / Y$ is weakly admissible. But $N/Y$ would then have only one horizontal segment, \ie
    \begin{equation*}
        \Fil^i(N/Y)=\begin{cases}
            N/Y,\,\, i\leq 0\\
            0,\,\, i\geq 1,
        \end{cases}
    \end{equation*}
which contradicts admissibility unless $N/Y = 0$. Thus $Y = N$.
\end{proof}

\begin{thm}\label{level 2 period cinjecture for Hpp}
The $\Hpp$-period conjecture holds at depth $2$.
\end{thm}

\begin{proof}
We need to prove that the evaluation map \[\tilde{\cP}^2_{\Hpp}(\gMi(\K))\to\cP_{\Hpp}(\gMi(\K))\]
is injective.
Let $\displaystyle\alpha_i = \int^{\varpi}_{x_i}\omega_i$ be $\Hpp$-periods of 1-motives $M_i$, with $x_i \in \tilde{\Hpp}(M_i)$ and $\omega_i \in \tilde{N}(M_i)$. Assume $\sum^n_{i=1} c_i\alpha_i = 0$ for some $c_i \in \K^u$. As it is shown in \cref{prop 4.1}, a linear combination of $\Hpp$-periods is again a $\Hpp$-period. More precisely, we can write
\[
0 = \int^{\Hpp}_x\omega, \quad \text{where } x = \sum^n_{i=1} x_i \in \tilde{\Hpp}(\bigoplus M_i),\quad \omega = \sum^n_{i=1} c_i\omega_i \in \tilde{N}(\bigoplus M_i).
\]
Let $M = \bigoplus M_i$; then there exists a finite extension of $\K$ over which $M$ is defined and has good reduction at $p$. We need to show that $(x\otimes\omega)_{\tilde{\cP}_{\Hpp}^2(M)}=0$. Throughout, we retain the notation introduced in \eqref{eq 6.2.8}.
Since $\omega \in \tilde{N}(M)$, there exists an isocrystal $N_{K_0'}$ defined over a finite unramified extension $K'_0$ such that $\omega \in \tilde{N}_{K_0'} \cap \TdRv(M)_{\K'}$ and $\tilde{N}_{K_0'} \subseteq \Tcrysv(\bar{M}) \otimes_{\W(k')} K'$. But $\Tcrysv(\bar{M})\otimes_{\W(k')}K'\cong \TdRv(M)_{K'}$ is a filtered isomorphism as discussed in \cref{cor: Hodge filtration and tangent space on Dieudonne module}, which induces the following filtration on $\tilde{N}_{K'_0}$ that is identical to the filtration induced as a subobject of $N_{K'_0}$.
\[
\Fil^i(\tilde{N}_{K_0'}) = \begin{cases}
\TdRv(M)_{K'} \cap \tilde{N}_{K_0'} & i \leq 0 \\
\coLie(G)_{K'} \cap \tilde{N}_{K_0'} & i = 1 \\
0 & i \geq 2
\end{cases}
\]
By \cref{lemma 5.6.2}, $\omega = F^n(\gamma)$ for some $\gamma \in \coLie(G) \cap \tilde{N}_{K_0'}$ and some positive integer $n$.

Since $x \in \tilde{\Hpp}(M)$, we may write $x = \sum b_i \otimes x_i$ 
with $x_i \in \tilde{\Tp}(M)$ and $b_i \in \Brigt=\bigcap_{m=1}^\infty \phi_{\mathrm{cris}}^m(\Bct)$. For each $b_i$, there exists $b_i' \in \Bct$ such that $\phi_{\mathrm{cris}}^n(b_i') = b_i$. Since $\phi_{\mathrm{cris}}$ is injective, it follows that $b_i' \in \Brigt$.

By \cref{cor 4.3.2}, we have
\[
\int^{\varpi}_x \omega = \int^{\mathrm{cris}}_x \omega = 0.
\]
Now consider
\[
\int^{\mathrm{cris}}_x \omega 
= \int^{\mathrm{cris}}_{b_1 \otimes x_1 + \dots + b_m \otimes x_m} F^n(\gamma)
= \sum_i \int^{\mathrm{cris}}_{x_i} b_i \otimes F^n(\gamma)=
\sum_i \int^{\mathrm{cris}}_{x_i} (\phi_{\mathrm{cris}}^n \otimes F^n)(b_i' \otimes \gamma).
\]

Since $F$ acts $\sigma$-semilinearly on $\tilde{N}_{K_0'}$ and $\phi_{\mathrm{cris}}$ $\phi_{cris}$ is a lift of Frobenius $\sigma$ over $K_0$, the map
\[
F \otimes \phi_{\mathrm{cris}} \colon N \otimes_{K_0} \Bct \longrightarrow N \otimes_{K_0} \phi_{\mathrm{cris}}(\Bct)
\]
is bijective and $\sigma$-semilinear. Therefore, by \cref{crystalline pairing and Frobenius}, we obtain
\[
\int^{\mathrm{cris}}_{x_i} (\phi_{\mathrm{cris}}^n \otimes F^n)(b_i' \otimes \gamma)
= \phi_{\mathrm{cris}}^n(b_i') \cdot \phi_{\mathrm{cris}}^n\left( \int^{\mathrm{cris}}_{x_i} \gamma \right)
= \phi_{\mathrm{cris}}^n\left( b_i' \cdot \int^{\mathrm{cris}}_{x_i} \gamma \right).
\]
Hence,
\[
\int^{\mathrm{cris}}_x \omega = \phi_{\mathrm{cris}}^n \left( \sum_i \int^{\mathrm{cris}}_{b_i' \otimes x_i} \gamma \right)
= \phi_{\mathrm{cris}}^n \left( \int^{\mathrm{cris}}_{\sum b_i' x_i} \gamma \right).
\]

Since $\phi_{\mathrm{cris}}$ is injective on $\Bcris$, it follows that
\[
0 = \int^{\mathrm{cris}}_{\sum b_i' x_i} \gamma
= \int^{\mathrm{cris}}_{\sum \phi^n(b_i') x_i} \gamma
= \int^{\mathrm{cris}}_x \gamma = \int^{\varpi}_x \gamma,
\]
where the second equality uses Frobenius-equivariance of crystalline integration \eqref{crystalline integration map}, and the last equality follows from \cref{cor 4.3.2}.

Since $\gamma \in \coLie(G)$, \cref{theorem p-adic integration pairing for M} implies
\[
\int^{\varpi}_x \gamma = \int^{\varphi}_x \gamma = 0.
\]
By \cref{comparing parings Hp and Hpp}, we have
\[
\int^{\varphi}_x \gamma = \int^{\varphi}_{\pi(x)} \gamma,
\]
with $\pi(x) \in \Hp(M)$. We now apply the $p$-adic subgroup theorem (\cref{p-adic subgroup theorem for Fontaine pairing}) to obtain an exact sequence
\begin{equation}\label{exact seq in proof of period conjecture}
0 \longrightarrow M_1 \longrightarrow M^n \longrightarrow M_2 \longrightarrow 0
\end{equation}
of $1$-motives over a finite extension of $\K'$, with $M_1$, $M_2$ admitting good reduction at $p$, $\pi(x) \in \Hp(M_1)$, $n \in \{1,2\}$, and $\Ann(\pi(x)) \subset \TdRv(M_2)_{\C_p}$.

It remains to show that $\omega \in \tilde{N}(M_2)$ and $x \in \tilde{\Hpp}(M_1)$. Since $\Ann(\pi(x)) \subset \TdRv(M_2)_{\C_p}$, we have $\gamma \in \TdRv(M_2)_{\C_p}$. But $\gamma \in \coLie(G)_{K'} \cap \tilde{N}_{K_0'}$, and it has non-zero slopes and the same holds for every power $F^n(\gamma)$, so $F^n(\gamma) = \omega \in \tilde{N}(M_2)$.

We can obtain an exact commutative diagram:
\[
\begin{tikzcd}
0 \arrow[r] & \Tp(M_1)\otimes\C_p(1) \arrow[d, hook] \arrow[r] & \Tp(M_1)\otimes\Bt \arrow[d, hook] \arrow[r] & \Tp(M_1)\otimes\C_p \arrow[d, hook] \arrow[r] & 0 \\
0 \arrow[r] & \Tp(M)\otimes \C_p(1) \arrow[r]                  & \Tp(M)\otimes\Bt \arrow[r]                   & \Tp(M)\otimes\C_p \arrow[r]                   & 0
\end{tikzcd}
\]
We can regard $x\in \Tp(M)\otimes\Brigt\into\Tp(M)\otimes\Bct\into\Tp(M)\otimes\Bt$, and we view $\pi(x)$ as an element in $\Tp(M_1)\otimes\C_p$. By the constructions of $\Hpp$ and $\tilde{\Hpp}$ (see \cref{definition Hpp} and \cref{definition of tilde of Hpp}), to show that $x\in\tilde{\Hpp}(M_1)$, it suffices to show that the image of $x$ lies in $\Tp(M_1)\otimes\Bt$. Fix a basis $\{v_i\}$ for $\Tp(M_1)\otimes\Bt$ and extend it to a basis $\{v_i, u_i\}$ for $\Tp(M)\otimes \Bt$. We can write 
\[
x=\sum_i b_iv_i + b'_iu_i, \text{ for some } b_i, b'_i\in \Brigt.
\]
Using the commutative diagram above, we have $x=v+k$ for some $v\in\Tp(M_1)\otimes\Bt$ and some $k\in\Tp(M)\otimes\C_p(1)$. The coefficient of $u_i$ in the representation of $k$ must be $b'_i$, as it is zero in the representation of $v$. Since $k\in\Tp(M)\otimes\C_p(1)$, \cref{Brigt cap Cp(1)} implies that $b'_i=0$. It follows that $x=\sum_i b_iv_i\in \Tp(M_1)\otimes\Brigt$, and thus $x\in\tilde{\Hpp}(M_1)$. 

We can now conclude that
\[
(x \otimes \omega)_{\tilde{\cP}^2_{\Hpp}} = 0 \quad \text{in } \tilde{\cP}^2_{\Hpp}(M),
\]
and therefore, by \cref{colim injective period conjecture}(2), the evaluation map
\[
\tilde{\cP}^2_{\Hpp}(\gMi(\K)) \to \cP_{\Hpp}(\gMi(\K))
\]
is bijective.

\end{proof}

\begin{thm}\label{depth 2 and 1 period cinjecture for Hp and hp}
The $\Hp$-period conjecture holds at depth~$2$, and the $\hp$-period conjecture holds at depth~$1$.
\end{thm}

\begin{proof}
As in the proof of \cref{level 2 period cinjecture for Hpp}, a general linear relation among $\Hp$-periods reduces to a single equation of the form
\[
\int^{\Hp}_x \omega = 0,
\]
for some $x \in \Hp(M)$ and $\omega \in \coLie(G)_{\bar{\K}}$. We assume such a relation and apply the $p$-adic subgroup theorem for motives (\cref{p-adic subgroup theorem for Fontaine pairing}) to $x \in \Hp(M)$. This yields an exact sequence
\begin{equation}\label{6.3.2 exact seq}
0 \to M_1 \to M^n \to M_2 \to 0
\end{equation}
of $1$-motives over a finite extension of $\K$, with good reduction at $p$, such that $x \in \Hp(M_1)$ and $\Ann(x) \subseteq \TdRv(M_2)_{\C_p} = \coLie(G^\natural_2)_{\C_p}$. This implies that $\omega \in \coLie(G^\natural_2)$.

The exact sequence \eqref{6.3.2 exact seq} induces the commutative diagram
\[
\begin{tikzcd}
0 \arrow[r] & \coLie(G_2) \arrow[r] \arrow[d, hook] & \coLie(G)^n \arrow[r] \arrow[d, hook] & \coLie(G_1) \arrow[r] \arrow[d, hook] & 0 \\
0 \arrow[r] & \coLie(G^\natural_2) \arrow[r]        & \coLie(G^\natural)^n \arrow[r]        & \coLie(G^\natural_1) \arrow[r]        & 0,
\end{tikzcd}
\]
and since $\omega \in \coLie(G^\natural_2) \cap \coLie(G)^n_{\bar{\K}}$, it follows that $\omega \in \coLie(G_2)_{\bar{\K}}$. Thus, in the space $\tilde{\cP}^2_{\Hpp}(M)$, we have
\[
(x \otimes \omega)_{\tilde{\cP}^2_{\Hpp}(M)} = 0.
\]
This completes the proof of the $\Hp$-period conjecture at depth~$2$.

For the $\hp$-period conjecture, the proof proceeds similarly. If $x \in \hp(M)$, then by \cref{cor 5.6.1}, in the exact sequence \eqref{6.3.2 exact seq} we must have $n = 1$. Therefore, the evaluation map
\[
\tilde{\cP}^1_{\hp}(\gMi) \to \cP_{\hp}(\gMi)
\]
is bijective.
\end{proof}

\begin{cor}
All relations among the $\hp$-periods of $\gMi$ are induced by bilinearity and functoriality. More precisely, the evaluation map
\[
\tilde{\cP}_{\hp}(\gMi) \to \cP_{\hp}(\gMi)
\]
is bijective.
\end{cor}

\begin{proof}
This follows directly from \cref{depth 2 and 1 period cinjecture for Hp and hp} and \cref{cor: P1=P for period conjecture}.
\end{proof}

\begin{remark}\label{remark 6.3.2}
Let $\alpha = \int^{\Hpp}_x \omega = 0$ be a vanishing $\Hpp$-period of $M$, where $x \in \tilde{\Hpp}(M)$ and $\omega \in \tilde{N}(M)$. Let $\pi\colon \tilde{\Hpp}(M) \to \Hp(M)$ denote the canonical projection. Then, combining the above argument with \cref{cor 5.6.1}, we see that if $\pi(x) \in \hp(M)$, then
\[
(x \otimes \omega)_{\tilde{\cP}^1_{\Hpp}(M)} = 0,
\]
which implies the existence of an exact sequence
\[
0 \to M_1 \to M \to M_2 \to 0
\]
such that $x \in \tilde{\Hpp}(M_1)$ and $\omega \in \tilde{N}(M_2)$. In particular, if $\Hp(M) = \hp(M)$, then the evaluation map
\[
\tilde{\cP}^1_{\Hpp}(M) \to \cP_{\Hpp}\langle M \rangle
\]
is injective, \ie all relation among $\Hpp$-periods of $M$ are induced by bilinearity and functoriality. Thus
\[
\tilde{\cP}^1_{\Hpp}(M) = \tilde{\cP}^2_{\Hpp}(M) = \cP_{\Hpp}\langle M \rangle.
\]
The same result applies to $\Hp$-periods.
\end{remark}

\begin{remark}\label{remark Hpp period does not hold at depth 1}
In general, we have $\tilde{\cP}^2_{\Hpp}(\gMi) \neq \tilde{\cP}^1_{\Hpp}(\gMi)$ and likewise $\tilde{\cP}^2_{\Hp}(\gMi) \neq \tilde{\cP}^1_{\Hp}(\gMi)$ (see \cref{example M[LtoGm]} in the next section). Therefore, by \cref{cor: P1=P for period conjecture}, the map
\[
\tilde{\cP}_{\Hpp}(\gMi) \to \cP_{\Hpp}(\gMi)
\]
is not injective. In other words, not all relations among $\Hpp$-periods (or $\Hp$-periods) are accounted for by bilinearity and functoriality. Rather, as shown in \cref{level 2 period cinjecture for Hpp} and \cref{depth 2 and 1 period cinjecture for Hp and hp}, all relations among the $\Hpp$-periods ($\Hp$-periods, respectively) are precisely captured by the depth-$2$ formal spaces $\tilde{\cP}^2_{\Hpp}$ (and $\tilde{\cP}^2_{\Hp}$, respectively).
\end{remark}

\begin{remark}
If one attempts to pair vectors in \( \Hp(M) \) with elements of \( \tilde{N}(M) \) or the de Rham realization \( \TdR(M) \) via p-adic integration map $\varpi_M$, the fact that \( \Hp(M) \) does not embed naturally into \( \TdR(M) \) gives rise to numerous \( p \)-adic period relations that cannot be captured by the methods developed in this paper.
\end{remark}


\section{Examples}\label{section: examples}

We illustrate our theory through two classes of examples: explicit $1$-motives and those associated to algebraic varieties.

\subsection{Explicit Motives}\label{subsec: explicit motives}

We begin with concrete $1$-motives defined over a number field $\K$ with good reduction at $p$.

\begin{example}[Periods of 0-motives]
Let $M = [L \to 0]$ be a $1$-motive with good reduction at $p$, where $r = \mathrm{rank}(L)$. Then:
\begin{itemize}
    \item The Hodge filtration of $L$ is given by:
    \[
    0 \to V(L) \to V(L) \to 0 \to 0.
    \]
    \item The crystalline realization is $\Tcrysv(\bar{M}) = \D(\bar{M}[p^\infty]) = \D((\underline{\Q_p/\Z_p}_k)^r) = (1_{FD})^r$, where $1_{FD}$ denotes the unit Dieudonné module equipped with $F=\sigma$ and the filtration
    \[
    \Fil^i\W(k)=\begin{cases}
    \W(k), & i\leq 0\\ 0,& i>0.
    \end{cases}
    \] Thus, all slopes are zero.
    \item $\tilde{N}(M) = 0$.
    \item $\hp(M) = 0$.
    \item $\Hp(M) = V(L)_{\C_p}$.
    \item All spaces of periods vanish:
    \[
    \cP_{\Hpp}(M) = \cP_{\Hp}(M) = \cP_{\hp}(M) = 0.
    \]
\end{itemize}
\end{example}

\begin{example}[Torus]
Let $M = [0 \to \G_m]$ be a $1$-motive with good reduction. Then:
\begin{itemize}
    \item The Hodge filtration of $M$ is:
    \[
    0 \to 0 \to \Lie(\G_m) \to \Lie(\G_m) \to 0,
    \]
    since $V(M) = \Ext^1(\G_m, \G_a)\ve = 0$.
    \item We have $\hp(M) = \log(\mup(\bar{\Q})) \otimes \Q = \log_{\mup}((1 + \fm_{\ocp}) \cap \bar{\Q}) \otimes_{\Z} \Q$ (recall \cref{logarith of p-divisible multiplicative} and \cref{definition hp}).
    \item Since $\mup$ is connected, we have $\Tp(M) = \tilde{\Tp}(M)$.
    \item $\Tcrysv(\bar{M}) = \D(\bar{M}[p^\infty]) = \D(\mup) = \Delta_1$, so all slopes are non-zero.
    \item As $V(M) = 0$, we have $\Hp(M) = \hp(M)$. By \cref{remark 6.3.2}, this implies
    \[
    \tilde{\cP}^1_{\Hpp}(M) = \tilde{\cP}^2_{\Hpp}(M) \cong \cP_{\Hpp}\langle M \rangle,
    \]
    i.e., all relations among $\Hpp$-periods of a torus are induced by bilinearity and functoriality. The same holds for $\Hp$-periods.
    \item As $\hp(M) = \Hp(M)$, we also have $\cP_{\hp}(M) = \cP_{\Hp}(M)$. Moreover, since all differential forms in $\coLie(\G_m)$ have non-zero slopes, we obtain
    \[
    \int^{\Hpp}_{x} \omega = \int^{\Hp}_{\pi(x)} \omega
    \]
    for any $\omega \in \tilde{N}(M) \subseteq \coLie(\G_m)_{\K^u}$ and $x \in \tilde{\Hpp}(M)$. This equality follows from \cref{comparing parings Hp and Hpp}, implying that $\cP_{\Hpp}(M) \subseteq \cP_{\Hp}(M)$.
\end{itemize}
\end{example}

\begin{example}[Kummer motive]\label{example M[LtoGm]}
Let $M = [\Z \to \G_m]$, $1 \mapsto g$, be a Kummer motive. If $g$ is a root of unity, the canonical exact sequence
\begin{equation}\label{exact seq in examples}
0 \to \G_m \to M \to \Z \to 0
\end{equation}
splits. Then:
\begin{itemize}
    \item the vector extension $G\nat$ is an extension of $\G_m$ by $\Hom(\Z,\G_a)=\G_a$, and by the structure theory of algebraic groups, we have $G\nat=\G_a\times\G_m$.
    \item Applying $\Tcrysv$ to \eqref{exact seq in examples} yields a split exact sequence
    \[
    0 \to 1_{FD} \to \Tcrysv(\bar{M}) \to \Delta_1 \to 0,
    \]
    giving slopes $0$ and $1$ of multiplicity $1$. Hence, $\tilde{N}(M)$ is supported in the filtered isocrystal associated with $\Delta_1$. 
    \item If the sequence \ref{exact seq in examples} splits, taking formal p-divisible groups yields a split exact sequence. Therefore, we have \[\hp(M)=\hp(\G_m)\oplus\hp(\Z)=\hp(\G_m),\]
since $\hp(\Z)=0$. In fact, this equality holds even if $g$ is not a root of unity. This is because $\Z[p^{\infty}]$ is \'etale, and consequently, $\hp(M)=\hp(\G_m)$.
    \item $V(M) = V(\Z) = \G_a$, hence $\Hp(M) = \hp(M) \oplus V(M)_{\C_p} = \hp(\G_m) \oplus V(\Z)_{\C_p}$.
    \item Consequently, $\cP_{\hp}\langle M \rangle = \cP_{\hp}\langle \G_m \rangle$, and $\cP_{\Hp}\langle \G_m \rangle \subset \cP_{\Hp}\langle M \rangle$.
\end{itemize}

Now, assume $\displaystyle\int^{\Hpp}_{x} \omega$ is an $\Hpp$-period of $M$, with $\omega \in \tilde{N}(M) \cap \coLie(\G_m)_{\bar{\K}}$ and $x \in \tilde{\Hpp}(M)$ such that $\pi(x) \in V(\Z)_{\C_p} \subset \Hp(M)$ and $x, \omega, \pi(x) \neq 0$. Then, by \cref{comparing parings Hp and Hpp},
\[
\int^{\Hpp}_{x} \omega = \int^{\Hp}_{\pi(x)} \omega = 0.
\]
The last equality follows because $\Ann(\pi(x))$ contains $\coLie(G)$, as $\pi(x) \in V(M)$. By \cref{level 2 period cinjecture for Hpp}, this implies $x \otimes \omega = 0$ in $\tilde{\cP}^2_{\Hpp}(M)$. Thus, there exists an exact sequence
\[
0 \to M_1 \to M^n \to M_2 \to 0
\]
with $n \leq 2$, $x \in \tilde{\Hpp}(M_1)$ and $\omega \in \tilde{N}(M_2)$. We claim $n = 2$. If $n = 1$, then $M_1$ must be one of $0$, $[0 \to \G_m]$, or $M$. The case $M_1 = 0$ is excluded as $x \neq 0$. If $M_1 = [0 \to \G_m]$, then $x\in\tilde{\Hpp}(M_1)=\tilde{\Hpp}(\G_m)$ and $\pi(x) \in \Hp(\G_m) = \hp(\G_m)$ which contradicts $\pi(x) \in V(\Z)_{\C_p}$, as both $\hp(\G_m)$ and $V(\Z)_{\C_p}$ are direct summands of $\Hp(M)$. The case $M_1 = M$ is excluded as $\omega \neq 0$. Hence $n = 2$, and:
\[
\tilde{\cP}^1_{\Hpp}(M) \neq \tilde{\cP}^2_{\Hpp}(M) \cong \cP_{\Hpp}\langle M \rangle.
\]
The latter identification is due to the fact that the $\Hpp$-period conjecture holds at depth 2 for $\langle M\rangle$ (\cref{level 2 period cinjecture for Hpp}).
An identical argument applies to $\Hp$-period $\displaystyle\int^{\Hp}_{\pi(x)}\omega=0$, showing:
\[
\tilde{\cP}^1_{\Hp}(M) \neq \tilde{\cP}^2_{\Hp}(M) \cong \cP_{\Hp}\langle M \rangle.
\]
By \cref{cor: P1=P for period conjecture}, this confirms the existence of nontrivial $\Hpp$- and $\Hp$-relations beyond bilinearity and functoriality.
\end{example}

\begin{prop}\label{prop 4.4.1}
Let $\alpha_1, \dots, \alpha_n$ be nonzero $\Hpp$-periods (resp., $\Hp$-periods) of $1$-motives $M_1, \dots, M_n$ over $\K$, and set $\cC_i = \langle M_i \rangle$. If $\Hom(\cC_i, \cC_j) = \Hom(\cC_j, \cC_i) = 0$ for $i \ne j$, then $\alpha_1, \dots, \alpha_n$ are $\K^u$-linearly (resp., $\bar{\Q}$-linearly) independent.
\end{prop}

\begin{proof}
We prove the case $n = 2$; the general case follows by induction. The formal period space at depth $2$ satisfies:
\[
\tilde{\cP}^2_{\Hpp}(M_1 \oplus M_2) = \tilde{\cP}^2_{\Hpp}(M_1) \oplus \tilde{\cP}^2_{\Hpp}(M_2).
\]
If $\lambda_1 \alpha_1 + \lambda_2 \alpha_2 = 0$ for $\lambda_i \in \K^u$, then $(\lambda_1\alpha_1+\lambda_2\alpha_2)_{\tilde{\cP}^2_{\Hpp}( M_1\oplus M_2)}$ vanishes in  $\tilde{\cP}^i_{\Hpp}( M_1\oplus M_2)$ by the validity of the $\Hpp$-period conjecture at depth 2 (\cref{level 2 period cinjecture for Hpp}). But \[\tilde{\cP}^2_{\Hpp}(M_1\oplus M_2)=\tilde{\cP}^2_{\Hpp}( M_1)\oplus\tilde{\cP}^2_{\Hpp}( M_2)\cong\cP_{\Hpp}\langle M_1\rangle\oplus\cP_{\Hpp}\langle M_2\rangle,\] thus $\lambda_i \alpha_i = 0$, and hence $\lambda_i = 0$, since $\alpha_i \ne 0$. The same argument applies to $\Hp$-periods.
\end{proof}

\begin{cor}
Let $A$ be an abelian variety over $\K$ with good reduction at $p$, and let $\alpha_1$ be a nonzero $\Hpp$-period (resp., $\Hp$-period) of $A$, and $\alpha_2$ a nonzero $\Hpp$-period (resp., $\Hp$-period) of $\G_m$. Then $\alpha_1, \alpha_2$ are $\K^u$-linearly (resp., $\bar{\Q}$-linearly) independent.
\end{cor}

\begin{proof}
Since there are no nontrivial morphisms between $\G_m$ and $A$, the conclusion follows from \cref{prop 4.4.1}.
\end{proof}

\subsection{1-motives associated with varieties}
Smooth varieties furnish further examples of $1$-motives and their associated $p$-adic periods. Following \cite{BS01}, let $X$ be an equidimensional variety over a field $K$ of characteristic $0$. Let $S \subset X$ denote its singular locus, $f: \tilde{X} \to X$ a resolution of singularities, and $\tilde{S}$ the reduced preimage of $S$. Consider a smooth compactification $\bar{X}$ of $X$ with boundary $Y = \bar{X} \setminus \tilde{X}$, and let $\bar{S}$ denote the Zariski closure of $\tilde{S}$.

The fpqc sheaf $T \mapsto \Pic(\bar{X}_T, Y_T)$ is representable by a $K$-group scheme locally of finite type, with $K$-points $\Pic(\bar{X}, Y)$ (\cite[Lemma 2.1]{BS01}). Let $\Pic^0(\bar{X}, Y)$ denote its identity component. Denote the smooth irreducible components of $Y$ by $Y_i$, and let $A(\bar{X}, Y) := \ker^0(\Pic^0(\bar{X}) \to \bigoplus_i \Pic^0(Y_i))$; this is an abelian variety, and we have the exact sequence:
\[
0 \to T(\bar{X}, Y) \to \Pic^0(\bar{X}, Y) \to A(\bar{X}, Y) \to 0,
\]
with $T(\bar{X}, Y)$ a torus. Thus, $\Pic^0(\bar{X},Y)$ should represent the semi-abelian part of the 1-motive associated with $X$, \ie $\W_{-1}(M)=\Pic^0(\bar{X},Y)$.

Let $\Div^0_{\bar{S}}(\bar{X}, Y)$ denote divisors $D$ supported on $\bar{S}$ disjoint from $Y$ and algebraically trivial in $\Pic^0(\bar{X}, Y)$. The kernel $\Div_{\tilde{S}/S}(\tilde{X}, Y)$ of $f_*: \Div_{\tilde{S}}(\tilde{X}) \to \Div_S(X)$ defines the group $\Div^0_{\bar{S}/S}(\bar{X}, Y) := \Div_{\tilde{S}/S}(\tilde{X}, Y) \cap \Div^0_{\bar{S}}(\bar{X}, Y)$.

\begin{defn}[\cite{BS01}]
The homological Picard $1$-motive of $X$ is
\[
\Pic^{-}(X) := [\Div^0_{\bar{S}/S}(\bar{X}, Y) \xrightarrow{u} \Pic^0(\bar{X}, Y)],
\]
where $u(D) = [D]$. The cohomological Albanese $1$-motive $\Alb^+(X)$ is its Cartier dual.
\end{defn}

\begin{example}
Let $C$ be a smooth curve and $D \subset C$ a subvariety of dimension $0$. For the homological Picard $1$-motive $\Pic^{-}(C)$ of $C$, $\W_{-1}(\Pic^{-}(C))$ is isomorphic to the generalized Jacobian $J(C)$ of $C$ (\cite[\S 1.8]{miyanishi_algebraic_2020}). Then the associated $1$-motive is
\[
[\Div^0(D) \to J(C)],
\]
where $\Div^0(D)$ is the group of degree-zero divisors on $D$ and $J(C)$ is the generalized Jacobian of $C$. For any  cohomology theory $H$, applying the long exact sequence for relative cohomology to the inclusion $C\to J(C)$ yields
\[
\begin{tikzcd}
H^0(J(C)) \arrow[r] \arrow[d, "f"] & H^0(D) \arrow[r] \arrow[d, Rightarrow, no head] & H^1(J(C), D) \arrow[r] \arrow[d, "h"] & H^1(J(C)) \arrow[d, "g"] \\
H^0(C) \arrow[r] & H^0(D) \arrow[r] & H^1(C, D) \arrow[r] & H^1(C)
\end{tikzcd}
\]
By the five lemma, the map $h$ is an isomorphism whenever both $f$ and $g$ are. According to \cite[Chapter V]{serrealgebraicgroupsandclaasfields}, this condition is satisfied for de Rham, singular, and étale cohomologies. In the case of crystalline cohomology over a field of characteristic $p \geq 3$, the result follows from \cite[Theorem B']{andreatta2005crystalline}. Consequently, all arithmetic information encoded in the $1$-motive can be transferred naturally to the pair $(C, D)$. In particular, our $\Q$-structures may be interpreted in terms of the homology of $(C, D)$, thereby allowing us to define $p$-adic periods for $(C, D)$ whenever the associated $1$-motive admits good reduction at $p$.
\end{example}

\begin{defn}
Let $X$ be a variety over a number field $\K$, with $\Pic^{-}(X)$ a $1$-motive of good reduction at $p$. An $\Hpp$-period (resp., $\Hp$-period, $\hp$-period) of $X$ is a $p$-adic number arising as an $\Hpp$-period (resp., $\Hp$-period, $\hp$-period) of $\Pic^{-}(X)$. Their spaces are denoted $\cP_{\Hpp}(X)$, $\cP_{\Hp}(X)$, and $\cP_{\hp}(X)$ respectively.
\end{defn}

When the $1$-motive associated with $X$ admits good reduction at $p$, the $p$-adic periods of $X$ arise from the $p$-adic integration pairing of the motive $\Pic^{-}(X)$. The vanishing of such periods admits a motivic classification, provided by \cref{level 2 period cinjecture for Hpp} and \cref{depth 2 and 1 period cinjecture for Hp and hp}, which describe the structure of all formal relations among them.

\begin{thm}
Let $X$ be a variety over $\K$ such that $M := \Pic^{-}(X)$ has good reduction at $p$. Let $\alpha$ be an $\Hpp$-period (resp., $\Hp$-period, $\hp$-period) of $X$, with
\[
\alpha = \int^{\Hpp}_x \omega = 0 \quad \text{(resp., } \int^{\Hp}_x \omega = 0,\, \int^{\hp}_x \omega = 0 \text{)}.
\]
Then there exists an exact sequence
\[
0 \to M_1 \to M^n \to M_2 \to 0
\]
of $1$-motives over a finite extension of $\K$ with good reduction at $p$, where $n \in \{1, 2\}$ (or $n = 1$ for $\hp$), $x \in \tilde{\Hpp}(M_1)$ (resp., $\Hp(M_1)$, $\hp(M_1)$), and $\omega \in \tilde{N}(M_2)$ (resp., $\Fil^1 \TdRv(M_2)_{\bar{\K}}$).
\end{thm}


\bibliographystyle{alpha}
\bibliography{refs.bib}

\end{document}